%% file: Weighted_multiplier_ideals.tex
\documentclass[11pt]{amsart}
\usepackage[utf8]{inputenc}
\usepackage{kantlipsum} 
\calclayout
\newif\ifdraft
\draftfalse
\input{macros.tex}

\DeclareMathOperator{\adj}{adj}
\DeclareMathOperator{\Gr}{Gr}
\newcommand{\remarkref}[1]{\hyperref[#1]{Remark~\ref*{#1}}}

\begin{document}

\vspace{\baselineskip}

\title{Weighted multiplier ideals of reduced divisors}

\author[Sebasti\'an~Olano]{Sebasti\'an~Olano}
\address{Department of Mathematics, University of Michigan,
Ann Arbor, MI 48109, USA}
\email{{\tt olano@umich.edu}}

\thanks{}


\begin{abstract} We use methods from birational geometry to study the Hodge and weight filtrations on the localization along a hypersurface. We focus on the lowest piece of the Hodge filtration of the submodules arising from the weight filtration. This leads to a sequence of ideal sheaves called weighted multiplier ideals. The last ideal of this sequence is a multiplier ideal (and a Hodge ideal), and we prove that the first is the adjoint ideal. We also study the local and global properties of weighted multiplier ideals and their applications to singularities of hypersurfaces of smooth varieties.
\end{abstract}

\maketitle

\section{Introduction}

Let $X$ be a smooth complex variety of dimension $n$. To an effective reduced divisor $D$ on $X$ one can associate a sequence of ideal sheaves $I_k(D)\subseteq \shO_X$, called the Hodge ideals of $D$ and studied in a series of papers \cite{hodgeideals},\cite{mustatapopa18},\cite{mustatapopa19}. They arise from the theory of M. Saito, which induces a Hodge filtration $F_{\bullet}\OX(*D)$ by coherent $\OX$-modules on $\OX(*D)$, the sheaf of functions with poles along $D$, seen as a left $\Dmod_X$-module. This $\Dmod$-module underlies the mixed Hodge module $j_*\QQ^H_{U}[n]$, where $j: U=X\smallsetminus D\into X$. Saito showed that the Hodge filtration is contained in the pole order filtration, that is,  $$F_k\OX(*D) \subseteq \OX((k+1)D)$$ for all $k\geq 0$. Using this, the Hodge ideal $I_k(D)$ is defined by $$F_k\OX(*D) = \OX((k+1)D) \otimes I_k(D).$$ The $\Dmod_X$-module $\OX(*D)$ is also endowed with a weight filtration $W_{\bullet}\OX(*D)$ by $\Dmod_X$-submodules. The Hodge filtration of these submodules satisfies $$F_kW_{n+l}\OX(*D) \subseteq F_k\OX(*D) \subseteq \OX((k+1)D),$$ and similarly we can define the weighted Hodge ideals by $$F_kW_{n+l}\OX(*D) = \OX((k+1)D) \otimes I^{W_l}_k(D).$$

The aim of this paper is to study the weighted Hodge ideals when $k=0$, using the methods of birational geometry and mixed Hodge modules. As the weight filtration is an increasing filtration, we have a chain of inclusions $$I^{W_0}_0(D)\subseteq I^{W_1}_0(D)\subseteq \cdots \subseteq I^{W_n}_0(D).$$ The last member in the chain is equal to the 0-th Hodge ideal, that is, $I^{W_n}_0(D)= I_0(D)$. This ideal is quite well understood; it is the multiplier ideal $$I_0(D) = \mathcal{J}(X, (1-\epsilon)D)$$ for $0<\epsilon\ll 1$ \cite{hodgeideals}*{Proposition 10.1}. The latter is a measure of log-canonicity, which means that using the above identification we obtain $I_0(D)= \OX$ if and only if the pair $(X,D)$ is log-canonical \cite{lazarsfeld2}*{Definition 9.3.9}. As the 0-th weighted Hodge ideals are contained in $I_0(D)$, by analogy we call them weighted multiplier ideals (ignoring the $(1-\epsilon)$ twist). This paper is the starting point of the project of studying weighted Hodge ideals; we pursue the general case in the sequel \cite{olano21b}.  \\

Weighted multiplier ideals can be computed using a log-resolution of singularities (see \propositionref{def}). Unlike $I_0(D)$, they are usually not defined as pushforwards of line bundles. The lowest non-trivial weight is $l=0$. For this value, the weighted multiplier ideal is the ideal of the divisor $D$: $$I_0^{W_0}(D) = \OX(-D).$$ The next weight is $l=1$, and the 1-st weighted multiplier ideal also admits an interpretation. 

\begin{intro-theorem}\label{thmadjoint} Let $X$ be a smooth variety and $D$ a reduced divisor. Then $$I_0^{W_1}(D) = \adj(D),$$ the adjoint ideal of $D$.  \end{intro-theorem}

The adjoint ideal is a measure of rationality of a singularity. More precisely, $\adj(D)= \OX$ if and only if $D$ is normal and has at most rational singularities \cite{lazarsfeld2}*{Proposition 9.3.48}. This ideal has been used in birational geometry, for instance, to study singularities of theta divisors \cite{einlazarsfeld97}. \theoremref{thmadjoint} gives a Hodge theoretic interpretation of the adjoint ideal. The proof of \theoremref{thmadjoint} relies on \propositionref{propositionadjoint}, and after the completion of this article, the author was informed that \propositionref{propositionadjoint} was already obtained in \cite{budursaito05}*{Proof of Proposition 3.5}. Together with \cite{budurthesis}*{\textsection 3.3}, this result had described an equivalent Hodge theoretic interpretation of the adjoint ideal to the one in \theoremref{thmadjoint}.\\

Weighted multiplier ideals form a sequence of ideals interpolating between the adjoint ideal and the multiplier ideal $I_0(D)$. In some sense, they measure and filter the ``distance'' between the pair $(X,D)$ having log-canonical singularities (equivalently $D$ being Du Bois  (e.g. \cite{kollarkovacs10})), and having canonical singularities (or equivalently $D$ having rational singularities).\\

Recall that weighted multiplier ideals satisfy $$I_0^{W_{l-1}}(D) \subseteq I_0^{W_{l}}(D).$$ The difference between these ideals is measured by a sheaf, denoted by $C_l$ (see \definitionref{differencesheaves}), that is supported on the singular locus of $D$ for $l\geq 2$. If $D$ has isolated singularities, these sheaves are skyscraper sheaves and we give a description of their length. For simplicity, assume $D$ has one isolated singularity $x\in D$, after possibly restricting to an open set. 

\begin{intro-theorem}\label{main} Let $g: \tilde{D}\to D$ be a log-resolution of singularities which is an isomorphism outside of $x$, and let $G$ be the exceptional divisor. For $l\geq 2$, $$\dim{(C_l)_x} = h^{0,n-l}(H^{n-2}(G)),$$  where $h^{0,n-l}(H^{n-2}(G))$ are the Hodge numbers of the mixed Hodge structure on $H^{n-2}(G)$. \end{intro-theorem}

In the case $l=n$, the dimension of $(C_n)_x$ can be identified with a Betti number of the dual complex $\Delta(G)$. The dual complex is a CW complex that can be defined for a variety with simple normal crossings (see for instance \cite{payne13}). 
\begin{intro-corollary}\label{corcn} Using the notation above, $$\dim{(C_n)_x} = h^{0,0}(H^{n-2}(G)) = \dim(W_0H^{n-2}(G)) = b_{n-2}(\Delta(G)).$$ \end{intro-corollary}
The last equality follows from an interpretation of the cohomology of the dual complex $\Delta(G)$ as the lowest weight of the cohomology of $G$ (see \cite{payne13}*{Section 3}).\\

The dimension of $(C_2)_x$ admits a simpler interpretation. Let $G = \cup{G_i}$ be the decomposition into irreducible components. Each $G_i$ is a smooth projective variety, endowed with the usual Hodge numbers.
\begin{intro-corollary}\label{corc2} Using the same notation, $$\dim{(C_2)_x} = h^{0,n-2}(H^{n-2}(G)) = \sum{h^{0,n-2}(G_i)}.$$ \end{intro-corollary}
The second equality follows from the description of the mixed Hodge structure on the cohomology of $G$ (see Section \ref{mhssnc}).\\

For example, when $X$ is a smooth threefold and $D$ a normal surface, $(C_3)_x$ and $(C_2)_x$ are the two skyscraper sheaves described by \theoremref{main}. We have a chain of ideals $$\adj(D) = I_0^{W_1}(D) \subseteq I_0^{W_2}(D) \subseteq I_0(D).$$ Using \corollaryref{corc2} we obtain that the first two ideals are equal if and only if all of the irreducible components of the exceptional divisor $G$ are rational curves. \corollaryref{corcn} says that the last two ideals coincide if and only if the dual complex $\Delta(G)$, which is a graph, has no cycles. The three ideals can be different, as seen in the example of $X = \mathbb{A}^3$, $D= ((x^3+y^3+z^3)(y^2z-x^3+4xz^2)+x^6yz = 0)$ and the singularity $x=(0,0,0)$ (see \exampleref{example}). In this case, $\dim{(C_2)_x} = 2$ and $\dim{(C_3)_x} = 8$. Note that this singularity is not weighted homogeneous. In such a case, that is, if $D$ has at most isolated weighted homogeneous singularities, there are at most two different weighted multiplier ideals (see \propositionref{whs}).\\

An immediate consequence of \theoremref{main} is that we can describe the difference between the adjoint ideal and the ideal $I_0(D)$ at a singular point $x\in D$ in terms of these dimensions: 
\begin{equation*} \sum_{l=2}^{n}{\dim{(C_{l})_x}} = \dim\Gr^0_F(H^{n-2}(G)) = h^{n-2}(G, \shO_G). \end{equation*}
This dimension is known, for instance, if the pair $(X,D)$ is log-canonical and $D$ does not have rational singularities. Ishii proved that under these assumptions $$h^{n-2}(G, \shO_G) = 1$$ \cite{ishii85}*{Proposition 3.7}. Thus, there exists a value $l$ such that $$\dim{(C_l)_x} = 1,$$ while the others are zero. A log-canonical singularity is called of type $(0, n-l)$ if $h^{0,n-l}(H^{n-2}(G))=1$ \cite{ishii85}*{Definition 4.1}, or equivalently if  $\dim{(C_l)_x} = 1$. For instance, if $X$ is a smooth threefold and $D$ a normal surface, a log-canonical singularity is of type $(0,1)$ if it is simple elliptic, and of type $(0,0)$ if it is a cusp singularity. \\ 

Weighted multiplier ideals also satisfy global results. When $X$ is a smooth projective variety, as a multiplier ideal the 0-th Hodge ideal satisfies the Nadel vanishing theorem (see \cite{lazarsfeld2}*{Theorem 9.4.8}). This theorem says that given an ample line bundle $L$, one has $$H^i(X, \omega_X(D) \otimes L\otimes I_0(D)) = 0$$ for $i\geq 1$. An easy application of a vanishing theorem of Saito \cite{saito90}*{Proposition 2.33} shows that the same holds for the weighted multiplier ideals (see \propositionref{kodairatypevanishing}). In the case when $D$ is an ample divisor, the vanishing still holds with $L=\shO_X$ under some restrictions. 

\begin{intro-theorem}\label{thmvanishingample} Let $X$ be a smooth projective variety and suppose $D$ is ample with at most isolated singularities. Then $$H^i(X, \omega_X(D)\otimes I_0^{W_l}(D)) = 0$$ for $i\geq 1$ and $l\geq 1$, except possibly if $l=1$, $i=1$ and $\dim{X}\geq 3$. 
\end{intro-theorem}
For an example where the vanishing does not hold when $l=1$ and $i=1$ see \remarkref{remarkvanishing}. \theoremref{thmvanishingample} follows from a general result that does not require $D$ to have isolated singularities (see \propositionref{vanishingample}).\\

The local and global results described above can be used to obtain results about the geometry of certain isolated singular points on hypersurfaces of $\PP^n$.
\begin{intro-corollary}\label{corind} Let $D\subseteq \PP^n$ be a hypersurface of degree $d$ with at most isolated singularities. Denote by $Z_l$ the scheme defined by $I_0^{W_l}(D)$. Then, $$H^0(\PP^n, \shO_{\PP^n}(k)) \onto H^0(\PP^n, \shO_{Z_l})$$ for $k\geq d-n-1$ if $l\geq 2$, and $k\geq d-n$ if $l=1$. \end{intro-corollary}
For instance, this gives a bound on the number of cusp singularities in a normal surface $D\subseteq \PP^3$, such that $(\PP^3, D)$ is log-canonical (see \corollaryref{proplcpoints}).\\

Finally, we study the behavior of weighted multiplier ideals of a pair $(X,D)$ and the pair $(H,D_H)$, where $H\subseteq X$ is a general hypersurface and $D_H$ is the restriction of $D$ to $H$. As $H$ is general, it is smooth and the ideal $I_0^{W_l}(D_H)$ is well defined. 
\begin{intro-theorem}\label{thmrestriction} Let $H\subseteq X$ be a general element of a base-point free linear system on $X$, and $D\restr{H} = D_H$.  For all $l\in \ZZ$, $$I_0^{W_l}(D)\cdot \shO_H = I_0^{W_l}(D_H).$$ \end{intro-theorem}

This result is the analogue of the Restriction Theorem for multiplier ideals when the hypersurface is general \cite{lazarsfeld2}*{Theorem 9.5.1}. For a more general statement, see \remarkref{remarkrestriction}. Using \theoremref{thmrestriction}, we obtain generic description of the sheaf $C_l$. More precisely, recall that the sheaves $C_l$ with $l\geq 2$ are supported on the singular locus of $D$. The rank of $C_l$ seen as a coherent sheaf over each of the irreducible components of $D_{\rm sing}$ is computed in terms of a resolution of singularities (see \propositionref{rankcl}).

\medskip

\noindent
{\bf Acknowledgements.}
I would like to thank Mihnea Popa for his constant support during the project, and Tommaso de Fernex, Lawrence Ein, Sándor Kovács, Mircea Musta\c{t}\u{a}, and Mingyi Zhang for very helpful discussions. Finally, I am thankful to the anonymous referee for valuable comments and especially for pointing out previously known results about the Hodge theoretic interpretation of the adjoint ideal.

\section{Preliminaries} 
\subsection{Mixed Hodge modules} In this section, we recall some facts about mixed Hodge modules and set up the notation we use throughout this paper.\\

Let $X$ be a smooth variety of dimension $n$. For a graded-polarizable mixed Hodge module $M$, we denote the underlying left regular holonomic $\Dmod_X$-module by $\Mmod$. Throughout this paper, we will only use left $\Dmod_X$-modules. We denote by $W_{\bullet}M$ the weight filtration on $M$ and by $$\gr^W_lM := W_lM/W_{l-1}M$$ the quotient, which is a polarizable Hodge module of weight $l$. We denote by $F_{\bullet}\Mmod$ the Hodge filtration. The de Rham complex is defined as: $$\DR(\Mmod) = \big[\Mmod\to\Omega^1_X\otimes_{\shO_X}\Mmod\to\cdots\to\omega_X\otimes_{\shO_X}\Mmod\big][n],$$ and the Hodge filtration of $\Mmod$ induces a filtration on this complex: $$F_p\DR(\Mmod) = \big[F_p\Mmod\to\Omega^1_X\otimes_{\shO_X}F_{p+1}\Mmod\to\cdots\to\omega_X\otimes_{\shO_X}F_{p+n}\Mmod\big][n].$$ The $p$-th subquotient of this filtration is the complex $$\gr^F_p\DR(\Mmod) = \big[\gr^F_p\Mmod\to\Omega^1_X\otimes_{\shO_X}\gr^F_{p+1}\Mmod\to\cdots\to\omega_X\otimes_{\shO_X}\gr^F_{p+n}\Mmod\big][n].$$ 

Let $D$ be a reduced effective divisor. The main mixed Hodge module we study in this paper is $j_*\QQ^H_{U}[n]$, where $j:U=X\smallsetminus D\into X$, whose underlying $\Dmod_X$-module is the sheaf of function with poles along $D$: $\shO_X(*D)$. To study $\shO_X(*D)$, we use a better understood mixed Hodge module, and the properties of pushforwards. For this, we fix a log-resolution of singularities of $(X,D)$, that is, a proper morphism $f:Y\to X$ such that $Y$ is smooth, it is an isomorphism over $U$, and $(f^*D)_{red} = E$ is a divisor with simple normal crossings. In this setup we have: \begin{equation}\label{pushforward}f_+\OY(*E) \cong H^0f_+\OY(*E) \cong \OX(*D)\end{equation} (see for example \cite{hodgeideals}*{Lemma 2.2}). As $E$ is a simple normal crossings divisor, the weight filtration of the $\Dmod_Y$-module $\OY(*E)$ can be described in terms of the strata. The lowest degree of the weight filtration is $n= \dim{Y}$, that is: $$W_{n-1}\OY(*E) = 0.$$ The lowest piece corresponds to the canonical Hodge module of $Y$: $$W_n\OY(*E) \cong \OY.$$ To describe the rest of the subquotients, we introduce the following very useful notation. Let $$E = \bigcup_{i\in I}E_i.$$ The variety $$E(l) = \bigsqcup_{\substack {J\subseteq I\\|J|=l}}E_J$$ is a smooth and possibly disconnected variety. If we denote $i_l: E(l) \to Y$ the map such that on each component is the inclusion, then \begin{equation}\label{isoks}\gr^W_{n+l}\OY(*E) \cong i_{l+}\shO_{E(l)}\end{equation} with a Tate twist (see \cite{ks21}*{Prop 9.2}). \\

In order to describe the weight filtration of a pushforward of a projective morphism, a useful tool is to use the spectral sequence associated to the weight filtration: \begin{equation}\label{ss1}E_1^{p,q}= H^{p+q}f_+(\gr^W_{-p}\OY(*E))\Rightarrow H^{p+q}f_+\OY(*E),\end{equation}  
which degenerates at $E_2$, and there is an isomorphism: $$E_2^{p,q}\cong \gr^W_{q}H^{p+q}f_+\OY(*E)$$ \cite{saito90}*{Proposition 2.15}.

\subsection{Mixed Hodge structure of a simple normal crossings variety}\label{mhssnc} The description of the mixed Hodge structure of a simple normal crossings variety is used in several instances throughout the paper. We state the complete description for the convenience of the reader.\\

Let $D$ be a simple normal crossings variety. We define the map $H^k\big(D(r)\big) \stackrel{\delta_r}{\longrightarrow} H^k\big(D(r+1)\big)$ in the following way: for a component $D_J$ of $D(r+1)$, we have the inclusions $\lambda_{i, r+1}: D_J \into D_{J\smallsetminus\{j_i\}}$, where $J= \{j_1,\ldots, j_{r+1}\}$ and $j_1<\cdots <j_{r+1}$. We obtain the map $$\lambda_{i, r+1}: D(r+1) \to D(r),$$ which is the corresponding inclusion on each component. The map $\delta_r$ is defined as the alternating sum of the pullbacks $$\delta_r := \sum_{i=1}^{r+1}{(-1)^{i+1}\lambda_{i,r+1}^*}$$ on $H^k\big(D(r))$. We obtain the complex $$ 0 \longrightarrow H^k\big(D(1)\big) \stackrel{\delta_1}{\longrightarrow} H^k\big(D(2)\big) \stackrel{\delta_2}{\longrightarrow} \cdots \stackrel{\delta_l}{\longrightarrow} 
H^k\big(D(l+1)\big) \stackrel{\delta_{l+1}}{\longrightarrow} \cdots,$$
in which all cohomologies have $\CC$-coefficients. The weight $k$ piece of the mixed Hodge structure on the cohomology of $D$ are the cohomologies of this complex. More precisely, we have  \begin{equation}\label{snchp}{\rm Gr}^W_k H^{k+l}(D) = \ker\delta_{l+1}/\im{\delta_l}.\end{equation}
The Hodge space $H^{p,q}\big({\rm Gr}^W_k H^{k+l}(D)\big)$ is obtained by applying first $H^{p,q}$ to the complex and then taking the cohomologies. The 0-th Hodge piece of the cohomologies has a description in terms of the cohomologies of the sheaf $\shO_D$, which says $$\Gr^F_0H^k(D) \cong H^k(D, \shO_D)$$ (see \cite{St83}*{(1.5)}).


\section{Characterizations}

\subsection{Definition} In this section, we introduce the weighted multiplier ideals using the theory of mixed Hodge modules. \\

A fundamental result by Saito about the Hodge filtration on $\OX(*D)$ states that  $$F_k\OX(*D) \subseteq \OX((k+1)D)$$ (see \cite{saito93}*{Proposition 0.9}). Using this result, Hodge ideals, denoted $I_k(D)$, are defined using the formula $$F_k\OX(*D) = I_k(D) \otimes\OX((k+1)D)$$ (see \cite{hodgeideals}*{Definition 9.4}). In this paper, we are interested in the case of $k=0$. We introduce and study weighted multiplier ideals, which are contained in the 0-th Hodge ideal $I_0(D)$, and are defined using that the weight filtration of $\OX(*D)$ satisfies $$F_kW_{n+l}\OX(*D) \subseteq F_k\OX(*D) \subseteq \OX((k+1)D)$$ for all $k\geq 0$. 

\begin{definition}[Weighted multiplier ideals] Let $X$ be a smooth complex variety and $D$ a reduced divisor. For $l\geq 0$, we define the ideal sheaf $I_0^{W_l}(D)$ on $X$ by the formula $$F_0W_{n+l}\OX(*D) = I_0^{W_l}(D)\otimes\OX(D).$$ We call $I_0^{W_l}(D)$ the \textit{$l$-th weighted multiplier ideal} of $D$. \end{definition} 

It is easy to see that there is in fact a chain of inclusions \begin{equation}\label{chain} I_0^{W_{1}}(D)\subseteq I_0^{W_{2}}(D)\subseteq \cdots\subseteq I_0^{W_{n-1}}(D)\subseteq I_0^{W_{n}}(D).\end{equation} Indeed, the weight filtration of $\OX(*D)$ is an increasing filtration, hence $$F_0W_{n+l}\OX(*D) \subseteq F_0W_{n+l+1}\OX(*D),$$ or equivalently $$\OX(D)\otimes I_0^{W_l}(D) \subseteq \OX(D)\otimes I_0^{W_{l+1}}(D).$$

\subsection{Simple normal crossings divisor}\label{snc} When $D$ is a simple normal crossings divisor, the ideals $I_0^{W_l}(D)$ can be identified with the ideals defining the strata of the divisor $D$.\\

Before giving the description, we note that the weighted multiplier ideals can be identified using the weight filtration of the sheaf $\omega_X(D)$ (for a description of the weight filtration, see \cite{voisin1}*{Section 8}).

\begin{proposition} Using the notation above $$\omY(D)\otimes I^{W_l}_0(D) \cong W_l\omY(D).$$ \end{proposition} 

\begin{proof} There is a filtered quasi-isomorphism $$\Omega^{\bullet}_X(\log{D})[n] \into \DR(\OY(*D))$$ (see for example \cite{ks21}*{Prop 9.1}). The complex $\Omega^{\bullet}_X(\log{D})$ has a weight filtration given by $W_l\Omega^{\bullet}_X(\log{D})$, and a Hodge filtration given by the trivial filtration of the complex (see \cite{elzeinetal}*{Section 3.4.1}). This means that $W_l\Omega^{\bullet}_X(\log{D})[n]$ and $\DR(W_{l+n}\OY(*D))$ are quasi-isomorphic, with the induced Hodge filtrations. Therefore, \begin{equation}\label{Wlomega}\gr^F_{-n}\DR(W_{n+l}\shO_X(*D)) \cong \gr^F_{-n}W_l\Omega^{\bullet}_X(\log{D})[n] \cong W_l\omega_X(D).\end{equation}
The lowest degree of the Hodge filtration of $W_{n+l} \shO_X(*D)$ is 0, hence \begin{equation*}\gr^F_{-n}\DR(W_{n+l}\shO_X(*D)) \cong \omega_X\otimes F_0\shO_X(*D) \cong \omega_X(D)\otimes I_0^{W_l}(D).\end{equation*} The result follows from combining the two isomorphisms above. \end{proof}

Using a local computation, we now show that the weighted multiplier ideals can be identified with the ideals of the intersections of the irreducible components of $D$.

\begin{proposition}\label{idealsforsnc} Let $X$ be a smooth complex variety and $D$ a simple normal crossings divisor. Then $$I^{W_l}_0(D) = \mathscr{I}_{D^{l+1}}, $$ the ideal sheaf corresponding to $D^{l+1}$, where $$D^j = \bigcup_{|J|=j}{D_J}.$$
\end{proposition}

\begin{proof} Suppose that around a point $p\in X$ we have coordinates $x_1, \ldots, x_n$ such that $D$ is defined by $x_1\cdots x_r=0$. Let $I=\{1,\ldots ,r\}$. The generators of $W_l\omY(D)$ when $0\leq l\leq r$ are $$\left\{\frac{\omega}{x_J}\right\}_{ J\subseteq I, \ |J|=l}$$ where $\omega$ is the standard generator of $\omega_X$, and for $J= \{j_1,\cdots ,j_l\}$, $x_J = x_{j_1} \cdots x_{j_l}$. If $l>r$, it has the same one generator as for $r$. These generators can be rewritten as $$\frac{\omega}{x_1\cdots x_r}\cdot x_{I\smallsetminus J}.$$
This means that around the point $$I^{W_l}_0(D) = \left <\left\{x_{i_1}\cdots x_{i_{r-l}} :\ i_j\in I\right\}\right >$$ if $0\leq l<r$, and $$I^{W_l}_0(D) = \OY$$ if $l\geq r$. \end{proof}

\subsection{General case} The weighted multiplier ideals for reduced divisors have an equivalent definition using log-resolutions that does not require the theory of mixed Hodge modules. After establishing this equivalent definition, we prove \theoremref{thmadjoint}.\\

First, let us recall the birational definition of $I_0(D)$. Let $X$ be a smooth complex variety and $D$ a reduced effective divisor. Let $f:Y\to X$ be a log-resolution of the pair $(X,D)$. As the 0-th Hodge ideal is the multiplier ideal $$I_0(D) = \mathcal{J}(X, (1-\epsilon)D),$$ for $0<\epsilon\ll 1$ by \cite{hodgeideals}*{Proposition 10.1}, we have that $$f_*\omega_Y(E) \cong \omega_X(D)\otimes I_0(D).$$ The weighted multiplier ideals admit a similar interpretation.
 
\begin{proposition}\label{def} Let $f:Y\to X$ be a log-resolution of the pair $(X,D)$. Then $$f_*W_k\omega_Y(E) \cong \omega_X(D) \otimes I_0^{W_k}(D).$$ \end{proposition}

\begin{remark} The weight filtration of $\omega_Y(E)$ is concentrated in degrees $0$ to $n$. For degree $n$, $W_n\omega_Y(E) = \omega_Y(E)$, therefore $$I_0(D) = I_0^{W_n}(D).$$ \end{remark}

To obtain this result, we use the fact that the weight filtration of $\OX(*D)$ is determined by the weight filtration of $\OY(*E)$ and the pushforwards of the weighted pieces. The following result is a description of the lowest Hodge piece of these pushforwards and is the key in the proof of \propositionref{def}.

\begin{lemma}\label{interpretation} There are isomorphisms: \begin{renumerate} \item $R^pf_* W_k\omega_Y(E) \cong F_0(H^pf_+(W_{k+n}\OY(*E))) \otimes \omega_X.$ \item $R^pf_*\omega_{E(k)} \cong F_0(H^pf_+(\gr^W_{n+k}\shO_Y(*E)))\otimes \omega_X.$ \end{renumerate} \end{lemma}

\begin{proof} Recall (\ref{Wlomega}): $$F_0W_{k+n}\OY(*E) \otimes\omega_X= \gr_{-n}^F\DR(W_{k+n}\OY(*E)) \cong W_k\omega_Y(E).$$ We use the compatibility of $\gr_{-n}^F\DR$ and the pushforwards that in this case says \begin{equation}\label{compatibility}\textbf{R}f_*\gr_{-n}^F\DR(W_{k+n}\OY(*E)) \cong \gr_{-n}^F\DR(f_+(W_{k+n}\OY(*E)))\end{equation} (see \cite{saito88}*{2.3.7}). The lowest degree of the Hodge filtration of $H^pf_+(W_{k+n}\OY(*E))$ is 0, therefore $\gr_{-n}^F\DR$ applied to this Hodge module is isomorphic to a sheaf in degree 0. Moreover, this sheaf is $$F_0(H^pf_+(W_{k+n}\OY(*E))) \otimes \omega_X.$$ Applying $\cohH^p$ to (\ref{compatibility}) and using that $\gr_{-n}^F\DR(H^pf_+(W_{k+n}\OY(*E)))$ is a sheaf, we obtain the first isomorphism (see e.g. \cite{schnellsurvey}*{Theorem 28.1}). The second isomorphism follows by similar arguments.
\end{proof}

The proof of \propositionref{def} is a combination of the definition of the weight filtration of $\OX(*D)$ and the previous lemma.

\begin{proof}[Proof of \propositionref{def}]The weight filtrations of $\OY(*E)$ and $\OX(*D)$ are related by: \begin{equation*}\label{weight} W_{n+k}\OX(*D) \cong \im[H^0f_+(W_{n+k}\OY(*E)) \stackrel{a_k}{\to} H^0f_+(\OY(*E))] \end{equation*} by \cite{saito90}*{Theorem 2.14} together with (\ref{pushforward}). Applying \lemmaref{interpretation} to the map $a_k$, we obtain \begin{equation*}F_0W_{n+k}\OX(*D) \otimes \omega_X =\gr^F_{-n}\DR W_{n+k}\OX(*D) \cong\im[f_*W_k\omega_Y(E) \to f_*\omega_Y(E)] \end{equation*} which is an inclusion. Finally, recall that $$F_0W_{n+k}\OX(*D) = \omega_X(D)\otimes I_0^{W_k}(D).$$
\end{proof}

Using \propositionref{def} we now give descriptions of the weighted multiplier ideals for low weights. For weight $l=0$, the weighted multiplier ideal is the ideal of the divisor $D$: $$I^{W_0}_0(D) = \OX(-D).$$ Indeed, this is a consequence of the isomorphism $$f_*\omega_Y \cong \omega_X.$$ 

When $l=1$ we have the following interpretation.

\begin{proposition}\label{propositionadjoint} Using the notation above, $$f_*W_1\omega_Y(E) = \omega_X(D)\otimes\adj(D).$$ \end{proposition}
\begin{remark} When $D$ is a simple normal crossings divisor, this was observed in \propositionref{idealsforsnc}. \end{remark}

\begin{proof} Consider the following map between the two short exact sequences:
\begin{diagram*}{2.5em}{2em}
\matrix[math] (m) { 0 & [-1.29em] \omega_Y & \omega_Y(\tilde{D}) & \omega_{\tilde{D}} & [-1.29em] 0 \\
										0 & [-1.29em]\omega_Y & W_1\omega_Y(E) & \omega_{E(1)} & [-1.29em] 0\\}; 
 \path[to] (m-1-1) edge (m-1-2);
 \path[to] (m-1-2) edge (m-1-3);
 \path[to] (m-1-3) edge (m-1-4);
 \path[to] (m-1-4) edge (m-1-5);
 \path[to] (m-2-1) edge (m-2-2);
 \path[to] (m-2-2) edge (m-2-3);
 \path[to] (m-2-3) edge (m-2-4);
 \path[to] (m-2-4) edge (m-2-5);
 \path[to] (m-1-2) edge node[auto]{$=$}   (m-2-2);
 \path[to] (m-1-3) edge  node[auto]{$\alpha$} (m-2-3);
 \path[into] (m-1-4) edge (m-2-4);

\end{diagram*}
where the inclusion on the right is the identity map on its corresponding summand, and the map $\alpha$ is the inclusion $$\omega_Y(\tilde{D}) \into W_1\omega_Y(E).$$

After applying $f_*$ to the diagram, we obtain:

\begin{diagram*}{2.5em}{2em}
\matrix[math] (m) { 0 &  \omega_X & \omega_X(D)\otimes\adj(D)  & f_*\omega_{\tilde{D}} & 0 \\
										0 &  \omega_X & f_*W_1\omega_Y(E) & f_*\omega_{E(1)} & 0\\}; 
 \path[to] (m-1-1) edge (m-1-2);
 \path[to] (m-1-2) edge (m-1-3);
 \path[to] (m-1-3) edge (m-1-4);
 \path[to] (m-1-4) edge (m-1-5);
 \path[to] (m-2-1) edge (m-2-2);
 \path[to] (m-2-2) edge (m-2-3);
 \path[to] (m-2-3) edge (m-2-4);
 \path[to] (m-2-4) edge (m-2-5);
 \path[to] (m-1-2) edge node[auto]{$=$}   (m-2-2);
 \path[to] (m-1-3) edge  node[auto]{$\beta$} (m-2-3);
 \path[into] (m-1-4) edge (m-2-4);

\end{diagram*}

To complete the proof, it is enough to show that given an irreducible component $E_i\subseteq E$ that is exceptional, $f_*\omega_{E_i}=0$, as by the snake lemma we obtain that $\beta$ is an isomorphism. An even stronger statement is well known. We include a proof of this simple case for completeness.\\

We can assume $f$ is a sequence of blowups of smooth centers, and let $g: X_1 \to X_0$ be one of such blowups such that it the image of $E_i$ in $X_1$ is a divisor, but it is not in $X_0$. Say $g$ is a blowup of $Z\subseteq X_0$, $F$ is the exceptional set of $g$ and $g': F \to Z$ is the restriction of $g$ to $F$. We have that $F$ is a $\PP^{n-1-s}$-bundle over $Z$, where $s=\dim{Z}$. As $\omega_F\restr{g^{-1}(z)} \cong \omega_{\PP^{n-1-s}}$ and $$H^0(\PP^{n-1-s}, \omega_{\PP^{n-1-s}})=0$$ we obtain that $$g'_*\omega_F = 0$$ and hence $$g_*\omega_F = 0.$$ Let $h:Y \to X_1$ be the composition of the sequence of blowups up until $X_1$ and $h': E_i \to F$ the restriction to $E_i$. The morphism $h'$ is birational, hence $h'_*\omega_{E_i} \cong \omega_F$, hence $$h_*\omega_{E_i} \cong \omega_F.$$ These two equations imply $$f_*\omega_{E_i} = 0.$$
\end{proof}

\begin{remark} After this article was completed, the author was informed that \propositionref{propositionadjoint} was implicit in the proof of \cite{budursaito05}*{Proposition 3.5}. \end{remark}

\begin{proof}[Proof of \theoremref{thmadjoint}] The results follows from combining \propositionref{def} and \propositionref{propositionadjoint}. \end{proof}

\section{Local Study} The difference between consecutive weighted multiplier ideals is measured by sheaves we introduce, which in most of the cases have their support contained in the singular locus of $D$. When $D$ has at most isolated singularities, these shaves are supported on the singular points, and so it makes sense to talk about their dimension at a point. \theoremref{main} states that the dimension of these sheaves are measured by invariants of the singularities, and we prove this result in Section \ref{higherdimensions}. Before giving the general proof, we examine some low dimensional cases for which the result follows from more elementary techniques.

\subsection{Measuring the difference between weighted multiplier ideals} Note that there is a short exact sequence $$0\to \omega_X(D) \otimes I_0^{W_{k-1}}(D) \to \omega_X(D) \otimes I_0^{W_k}(D) \to C_k\to 0$$ which is obtained by applying the functor $\gr_{-n}^F\DR$ to the short exact sequence given by the weight filtration $$0\to W_{n+k-1}\OX(*D) \to W_{n+k}\OX(*D) \to \gr^W_{n+k}\OX(*D) \to 0.$$

\begin{definition}\label{differencesheaves} We define $$C_k := \gr_{-n}^F\DR(\gr^W_{n+k}\OX(*D)).$$ Denote $$\tilde{C_k}:= C_k\otimes \omega_X(D)^{-1}.$$  \end{definition}

These sheaves measure the difference between consecutive weighted multiplier ideals. For $k=0$ and $k=1$ they have an easy description. In the first case, it is easy to see that $C_0 = \omega_X$. For $k=1$ we have \begin{equation} C_1\cong f_*\omega_{\tilde{D}}. \end{equation} Indeed, it follows from the proof of \theoremref{thmadjoint}, where it was observed that the sequence \begin{equation*}0\to \omega_X(D) \otimes I_0^{W_0}(D) \to \omega_X(D) \otimes I_0^{W_1}(D) \to C_1\to 0 \end{equation*} can be identified with the well known sequence: \begin{equation*} 0\to \omega_X \to \omega_X(D)\otimes \adj(D) \to f_*\omega_{\tilde{D}} \to 0. \end{equation*} For the other values of $k$ we have the following.

\begin{remark}\label{support} For $k\geq 2$, $\Supp{C_k}\subseteq D_{sing}$. This follows from the fact that the adjoint ideal is supported on the singular locus of $D$. We can also use (\ref{ss1}) for $p = -n-k$ and $q= n+k$, noting that $$f(E(k))\subseteq D_{sing}$$ for $k\geq 2$, together with (\ref{pushforward}) and (\ref{isoks}).
\end{remark}

\subsection{Curves in surfaces} Let $X$ be a smooth surface and $D$ a curve. In this case, we can give a complete description of the sheaves $C_k$. Using the notation of the previous section, $n=2$, which means that we are only missing a description of $C_2$. This sheaf measures the difference between the ideal $I_0(D)$ and the adjoint ideal of $D$. By \remarkref{support}, $C_2$ is a skyscraper sheaf. It is easy to see that the dimension of $C_2$ on a singular point of $x\in D$ is $$\dim{(C_{2})_x} = \#{g^{-1}}(x) - 1, $$ where $g:\tilde{D}\to D$ is the restriction of $f$ to $\tilde{D}$. To conclude this, we use an argument similar to the one we present in the next section. That argument is more intricate, and this result follows directly from it adjusting it to this dimension.

\subsection{Surfaces in threefolds}\label{surfaces} It is illustrative to give a proof of \theoremref{main} in the case of a smooth threefold $X$ and $D$ a surface with at most isolated singularities. In the next section, a general proof is given using the theory of mixed Hodge modules. For the proof in this section, we can recover some of those results without the use of such theory. In this way, this proof helps to illustrate what is happening in higher dimension.\\

Let $X$ be a smooth threefold and $D$ a normal surface. As $D$ has isolated singularities, $C_k$ is a skyscraper sheaf for $k\geq 2$ by \remarkref{support}. In this case, we can give descriptions at a singular point of $D$ of the sheaves $C_2$ and $C_3$. These sheaves measure the differences between the one potential ``new'' ideal $I_0^{W_2}(D)$ and the extremes in the following chain of inclusions: $$I^{W_1}_0(D)= \adj(D) \subseteq I^{W_2}_0(D) \subseteq I_0(D).$$ 
In order to describe them, we first fix some notation. Let $f:Y\to X$ be a log-resolution of $(X,D)$ and denote $f^{-1}(D)_{red} = E = \tilde{D} + F$. For simplicity, we assume that $x\in D$ is the only singular point of $D$, and $G:= f^{-1}(x)_{red} \cap\tilde{D}$. The curve $G$ is a simple normal crossings variety, and it is $g^{-1}(x)_{red}$, where $g:\tilde{D}\to D$ is defined by restricting $f$ to $\tilde{D}$. Recall that for a simple normal crossings variety $G$, we can define the dual complex $\Delta(G)$. For the definition of dual complex see \cite{payne13}*{Section 2}. The statement of \theoremref{main} in this dimension is the following:
\begin{renumerate} \item $\dim{(C_{2})_x} = h^{0,1}(H^1(G))= h^{0,1}(G(1))= \sum{g(G_i)}.$ \item $\dim{(C_{3})_x} = h^{0,0}(H^1(G)) = b_1(\Delta(G)).$ \end{renumerate} 
The second equalities are explained in the proof. This result has the following immediate consequence.

\begin{corollary}\label{corsurfaces} The inclusion $$I_0^{W_1}(D)_x \subseteq I_0^{W_2}(D)_x$$ is an isomorphism if and only if every component of $G(1)$ is rational. The inclusion $$I_0^{W_2}(D)_x \subseteq I_0(D)_x$$ is an isomorphism if and only if the dual complex associated to $G$ is contractible (i.e has no cycles). \end{corollary}
\begin{remark} Note that this condition is determined by a minimal resolution of $D$, since on a different log-resolution we have only extra rational exceptional curves, which do not form new cycles. \end{remark}

\begin{example}\label{example}The two inclusions in \corollaryref{corsurfaces} can be strict, as is the case of $X=\mathbb{A}^3$ and $D\subseteq X$ defined by $$(x^3+y^3+z^3)(y^2z-x^3+4xz^2) + x^6yz = 0.$$ The surface $D$ has an isolated singularity at $(0,0,0)$. A log-resolution of singularities is obtained by first blowing up this point, and then blowing up the 9 singular points in the exceptional set. The 9 points correspond to the intersection points of the cubics $(x^3+y^3+z^3=0)\subseteq \PP^2$ and $(y^2z-x^3+4xz^2=0)\subseteq\PP^2$, which intersect transversely. Each of these points is a node in the strict transform of $D$, and by blowing them up we obtain a log-resolution of singularities. Using this resolution we obtain $$\dim{(C_{2})_{(0,0,0)}} = 2,$$ as there are two genus one exceptional curves isomorphic to the cubics described above, and the other exceptional curves are rational. Also, using the explicit description of the dual complex we obtain $$\dim{(C_{3})_{(0,0,0)}} = 8.$$ Surfaces in $\mathbb{A}^3$ defined by an equation $$f_1\cdots f_r +M = 0$$ where $f_i=0$ is a curve in $\PP^2$, and at least one is non-rational, and $M$ is a high degree monomial such that the equation has isolated singularities, give other examples of a pair $(X,D)$ for which the three ideals are different. \end{example}

The proof of \theoremref{main} for $n=3$ is split into \propositionref{I0w2/I0w1}, \propositionref{h01g1}, \propositionref{I0w2/I0} and \propositionref{h00g1}. Note that after possibly restricting to an affine open set of $X$ which contains the singular point $x$, we can assume that $X$ is a projective variety. Indeed, there is an open set around $x$ which has a smooth projective compactification $\bar{X}$. Let $\bar{D}$ be the closure of $D$ in $\bar{X}$. Consider a log-resolution of $(\bar{X}\smallsetminus x, \bar{D}\smallsetminus x)$ given by a sequence of blow ups with centers over the singular locus of $\bar{D}\smallsetminus x$. By blowing up the same sequence of centers over $\bar{X}$, we obtain a map $X_1\to \bar{X}$. Let $D_1$ be the strict transform of $\bar{D}$. By construction, the map is an isomorphism over $(X,D)$, and $D_1$ has only one isolated singularity corresponding to $x\in D$. We replace $(X,D)$ with $(X_1,D_1)$.\\

\begin{proposition}\label{I0w2/I0w1} There is an isomorphism, $$(C_2)_x\cong \coker{\big[H^{0,1}(F(1)) \stackrel{\alpha}{\to} H^{0,1}(E(2))\big]^*}$$ where the map is given by restriction. \end{proposition}

\begin{proof} By \cite{voisin1}*{Proposition 8.32}, $$W_l\omega_Y(E) / W_{l-1}\omega_Y(E) \cong \omega_{E(l)},$$ identifying $\omega_{E(l)}$ with the pushforward. Using the short exact sequence $$ 0 \to \omega_Y \to W_1\omega_Y(E) \to \omega_{E(1)} \to 0$$ we obtain 
\begin{equation}\label{R1fW1} R^1f_*(W_1\omega_Y(E))_x \cong (R^1f_*\omega_{E(1)})_x \cong H^{2,1}(F(1))_x\cong H^{0,1}(F(1))_x^*.\end{equation}
\begin{equation}\label{R2fW1} R^2f_*(W_1\omega_Y(E))_x \cong (R^2f_*\omega_{E(1)})_x \cong H^{2,2}(F(1))_x \cong H^0(F(1))_x^*. \end{equation}
The second isomorphism in each of the equations follows from Grauert–Riemenschneider vanishing theorem, and the last one is given by Poincaré duality.\\

Pushing forward the short exact sequence $$ 0 \to W_1\omega_Y(E) \to W_2\omega_Y(E) \to \omega_{E(2)} \to 0$$ we obtain: \begin{equation}\label{i0w1i0w2} 0\to I_0^{W_1}(D)\otimes \omega_X(D) \to I_0^{W_2}(D)\otimes\omega_X(D) \to f_*\omega_{E(2)} \stackrel{\delta}{\to} R^1f_*(W_1\omega_Y(E)) \to R^1f_*(W_2\omega_Y(E)) \end{equation} As every component of $E(2)$ is exceptional, $$(f_*\omega_{E(2)})_x \cong H^{1,0}(E(2))_x.$$ Composing $\delta$ with the isomorphism (\ref{R1fW1}) on the stalk of $x$, we obtain that $$(C_2)_x\cong \ker{\big[H^{1,0}(E(2))_x \to H^{2,1}(F(1))_x\big]},$$ and the map is the Gysin map. The kernel is isomorphic to the cokernel of the dual map, that is $$(C_2)_x \cong \coker{\big[H^{0,1}(F(1))_x \to H^{0,1}(E(2))_x\big]^*}.$$  \end{proof}

\begin{proposition}\label{h01g1} Using the notation of \propositionref{I0w2/I0w1}, $$\dim{(C_2)_x} = h^{0,1}(G(1)).$$ \end{proposition}

\begin{proof} We prove first that \begin{equation}\label{isoh01f} H^{0,1}(F(1)) \cong H^{0,1}(F(2)). \end{equation} As $F(3)$ is 0-dimensional, we have a short exact sequence $$0 \to H^{0,1}\Gr_1^WH^1(F) \to H^{0,1}(F(1)) \to H^{0,1}(F(2)) \to 0.$$ Indeed, the last map is surjective since the cokernel is isomorphic to $H^{0,1}\Gr_1^WH^2(F)$. Using the long exact sequences of the pairs $(Y,F)$ and $(X, x)$, together with the identification of $Y\smallsetminus F$ and $X\smallsetminus x$, we obtain the following isomorphisms: $$\Gr_1^WH^2(F) \cong \Gr_1^WH^3_c(Y\smallsetminus F) \cong \Gr_1^WH^3(X)=0$$ where the last equality follows from the fact that $X$ is smooth. Moreover, $\Gr_1^WH^1(F) = 0$. Indeed, using again the long exact sequence associated to the pair $(X,x)$, and applying $\Gr_1^W$ we obtain the exact sequence:  $$0\to \Gr_1^WH^1_c(X\smallsetminus x) \to H^1(X) \to \Gr_1^WH^1(x) \to \Gr_1^WH^2_c(X\smallsetminus x) \to \Gr_1^WH^2(X).$$ As $X$ is smooth, we obtain $$ \Gr_1^WH^1_c(X\smallsetminus x) \cong H^1(X)$$ and $$\Gr_1H^2_c(X\smallsetminus x) = 0.$$

 Using the long exact sequence associated to the pair $(Y,F)$, and applying $\Gr_1^W$ we obtain the exact sequence: $$0\to \Gr_1^WH^1_c(Y\smallsetminus F) \to H^1(Y) \to \Gr_1^WH^1(F) \to \Gr_1^WH^2_c(Y\smallsetminus F).$$ As $Y\smallsetminus F\cong X\smallsetminus x$ the claim follows from the fact that $$H^1(X) \cong H^1(Y),$$ as they are birational, and the Hodge numbers $h^{1,0}$ coincide for both varieties. \\

From \propositionref{I0w2/I0w1}, we have $$\dim{(C_2)_x} = h^{0,1}(E(2)) - h^{0,1}(F(1)) + \dim{\ker{\alpha}}.$$ Note that $\ker{\alpha}\subseteq H^{0,1}\Gr_1^WH^1(F) =0$, as a class in the kernel must map to 0 in $H^{0,1}(F(2))$ as well. By the equality $E(2) = G(1) \sqcup F(2)$, together with (\ref{isoh01f}) we obtain $$ \dim{(C_2)_x} = h^{0,1}(G(1)).$$ By Section \ref{mhssnc}, $$h^{0,1}(G(1)) = h^{0,1}(H^1(G)).$$
\end{proof}

\begin{proposition}\label{I0w2/I0} There is an isomorphism, $$(C_3)_x\cong \coker{\big[H^0(E(2))/H^0(F(1)) \to H^0(E(3)) \big]^*.}$$ \end{proposition}

\begin{proof}Pushing forward the short exact sequence $$0\to W_2\omega_Y(E) \to \omega_Y(E) \to \omega_{E(3)} \to 0,$$ we obtain \begin{equation} 0\to I_0^{W_2}(D)\otimes\omega_X(D) \to I_0(D)\otimes\omega_X(D) \to f_*\omega_{E(3)} \to R^1f_*(W_2\omega_Y(E)) \to 0, \end{equation} as $R^1f_*\omega_Y(E) = 0$ by \cite{hodgeideals}*{Corollary 12.1}. From the sequence, it follows that $$(C_3)_x\cong \ker{\big[H^0(E(3))_x} \stackrel{\eta}{\to} (R^1f_*(W_2\omega_Y(E)))_x\big].$$ The next terms in the long exact sequence (\ref{i0w1i0w2}) are: $$R^1f_*(W_1\omega_Y(E)) \stackrel{\beta}{\to} R^1f_*(W_2\omega_Y(E)) \to R^1f_*\omega_{E(2)} \to R^2f_*W_1\omega_Y(E) \to 0.$$ In \propositionref{h01g1} it was shown that $\ker{\alpha} = 0$, and this is equivalent to $\beta = 0$. The sequence above, after composing with the isomorphism (\ref{R2fW1}) is isomorphic to the following short exact sequence: $$ 0\to (R^1f_*(W_2\omega_Y(E)))_x \to H^{1,1}(E(2))_x \to H^{2,2}(F(1))_x \to 0.$$ As the first map is injective, $$(C_3)_x\cong \ker{\big[H^0(E(3))_x \to H^{1,1}(E(2))_x\big]},$$ where the map is given by the composition $H^0(E(3))_x \stackrel{\eta}{\to} (R^1f_*(W_2\omega_Y(E)))_x \to H^{1,1}(E(2))_x$. Equivalently $$(C_3)_x\cong \coker{\big[H^0(E(2)) \to H^0(E(3))\big]^*}$$ by taking the dual map. Note that the image of $H^0(F(1)) \to H^0(E(2))$ lies in the kernel of the map above (which follows from dualizing the short exact sequence), and the map is injective on this quotient. The result follows.
\end{proof}

\begin{proposition}\label{h00g1} Using the notation of \propositionref{I0w2/I0},  $$\dim{(C_3)_x} = b_1(\Delta(G)).$$ \end{proposition}

\begin{proof} The cohomology of $\Delta(G)$ is isomorphic to the lower weight of the cohomology of $G$, that is $$H^k(\Delta(G)) \cong W_0H^k(G)$$ (see for example \cite{payne13}*{Section 3.}). Recall that $$W_0H^1(G) \cong \coker{[H^0(G(1)) \to H^0(G(2))]}.$$ The kernel of this map is isomorphic to $\CC$, as $G$ is connected. Hence, $$b_1(\Delta(G)) = b_0(G(2)) - b_0(G(1)) + 1.$$

There is an exact sequence \begin{equation}\label{exactF}0\to \CC \to H^0(F(1)) \to H^0(F(2)) \to H^0(F(3)) \to 0.\end{equation} Indeed, this follows because $$W_0H^j(F) \cong W_0H_c^{j+1}(Y\smallsetminus F) \cong W_0H^{j+1}(X) = 0$$ for $j=1,\ 2$,  by a similar argument as the one used in \propositionref{h01g1}, and the fact that $F$ is connected. Applying \propositionref{I0w2/I0} and (\ref{exactF}), we obtain \begin{equation*} \begin{split} \dim(C_3)_x &= \dim{(\coker{\big[H^0(E(2))/H^0(F(1)) \to H^0(E(3)) \big]^*})}\\ &= b_0(G(2)) + b_0(F(3)) - (b_0(F(2)) + b_0(G(1)) - b_0(F(1)))\\ &= b_0(G(2)) - b_0(G(1)) + 1\\& = b_1(\Delta{G}) \end{split} \end{equation*} \end{proof}

\subsection{Higher dimension}\label{higherdimensions} In this section we give a general identification of the sheaves $C_k$ that generalize the descriptions given in Section \ref{surfaces}. This identification works for a pair of a smooth variety $X$ and a reduced divisor $D$, with no restrictions on the singularities of $D$. Next, we give a proof of \theoremref{main}. \\

Let $f:Y\to X$ be a log-resolution and $E:=f^{-1}(D)_{red}$.
\begin{proposition}\label{exactsequence} There is an exact sequence $$0\to C_k \to f_*\omega_{E(k)} \to R^1f_*\omega_{E(k-1)} \to\cdots\to R^{k-1}f_*\omega_{E(1)} \to 0$$ \end{proposition}

\begin{proof} Consider the spectral sequence (\ref{ss1}). The maps $E_1^{-n-k+j-1, n+k} \to E_1^{-n-k+j,n+k} \to E_1^{-n-k+j+1,n+k}$ correspond to $$ H^{j-1}f_+(\gr^W_{n+k-j+1}\shO_Y(*E)) \to H^jf_+(\gr^W_{n+k-j}\shO_Y(*E)) \to H^{j+1}f_+(\gr^W_{n+k-j-1}\shO_Y(*E)).$$ For $j\geq 1$, $H^jf_+\shO_Y(*E) = 0$ by (\ref{pushforward}), thus the complex above is exact. Applying \lemmaref{interpretation}, we obtain the exact sequence $$R^{j-1}f_*\omega_{E(k-j+1)} \to R^jf_*\omega_{E(k-j)} \to R^{j+1}f_*\omega_{E(k-j-1)}.$$  Finally, for $j=0$ we have that $$E_2^{-n-k, n+k} \cong \gr^W_{n+k}H^0f_+\shO_Y(*E) \cong \gr^W_{n+k}\shO_X(*D).$$ Applying \lemmaref{interpretation} again, the corresponding exact sequence is then $$0\to C_k \to f_*\omega_{E(k)} \to R^1f_*\omega_{E(k-1)}.$$ \end{proof}

\begin{corollary}\label{differencei0} There is an isomorphism $$C_k\cong \ker[f_*\omega_{E(k)} \to R^1f_*\omega_{E(k-1)}].$$ \end{corollary}

\begin{remark} The map  $$f_*\omega_{E(k)} \to R^1f_*\omega_{E(k-1)}$$ can be described in simple terms. Consider the short exact sequences: $$0\to W_{k-1}\omega_Y(E) \to W_k\omega_Y(E) \to \omega_{E(k)} \to 0$$ and $$0\to W_{k-2}\omega_Y(E) \to W_{k-1}\omega_Y(E) \to \omega_{E(k-1)} \to 0.$$  By pushing them forward, we obtain the following maps: 
\begin{equation*}\label{deltak} f_*\omega_{E(k)} \to R^1f_*W_{k-1}\omega_Y(E) \end{equation*} and 
\begin{equation*}\label{epsilonk} R^1f_*W_{k-1}\omega_Y(E) \to R^1f_*\omega_{E(k-1)}. \end{equation*} The map $$f_*\omega_{E(k)} \to R^1f_*\omega_{E(k-1)}$$ is the composition of these two maps. \end{remark}

From now on we assume that $D$ has only one isolated singularity $x\in D$. As in Section \ref{surfaces}, we assume that $X$ is a smooth projective variety.

\begin{proof}[Proof of \theoremref{main}] First, we show that dimension $$h^{0,n-l}(H^{n-2}(G))$$ does not depend on the log-resolution of singularities that is an isomorphism outside of $x$. It is enough to show this is satisfied for two resolution of singularities $g_1: D_1\to D$ and $g_2: D_2\to D$, such that there is a morphism $h: D_1\to D_2$, with $g_1 = g_2\circ h$. We denote by $G_i$ the exceptional divisor of $g_i$. There is an exact sequence of mixed Hodge structures $$H^{n-3}(G_1)\to H^{n-2}(D_2)\to H^{n-2}(D_1)\oplus H^{n-2}(G_2) \to H^{n-2}(G_1) \to H^{n-1}(D_2)$$ as the map $h$ has discriminant $G_2$ (see \cite{PS}*{Corollary 5.37}). After taking $\Gr_F^0\Gr^W_{n-l}$, for $l\geq 3$, we obtain the sequence $$0\to H^{0,n-l}(H^{n-2}(G_2)) \to H^{0,n-l}(H^{n-2}(G_1))\to 0.$$ For $l=2$ we obtain $$0\to H^{0,n-2}(D_2)\to H^{0,n-2}(D_1)\oplus H^{0,n-2}(H^{n-2}(G_2)) \to H^{0,n-2}(H^{n-2}(G_1))\to 0.$$ As $D_1$ and $D_2$ are birational, $h^{0,n-2}(D_1)= h^{0,n-2}(D_2)$, hence $$h^{0,n-2}(H^{n-2}(G_1)) = h^{0,n-2}(H^{n-2}(G_2)).$$

Let $f:Y\to X$ be a log-resolution that is an isomorphism outside of $x$, and $E:=f^{-1}(D)_{red}$. This defines a log-resolution of singularities $g: \tilde{D}\to D$ by restriction, which is an isomorphism outside of $x$. We use the spectral sequence (\ref{ss1}) for the constant map from $X$ to a point. In this case it says \begin{equation}\label{ss2} E_1^{-n-l,q} = H^{q-n-l}(X, \DR(\gr^W_{n+l}\OX(*D))) \Rightarrow H^{q-l}(U,\CC),\end{equation} as $\DR(\OX(*D)) \cong \textbf{R}j_*\CC_U[n]$, where $j:U = X\smallsetminus D\into X$. We also have the isomorphism $$E_2^{-n-l,q}\cong \Gr^W_{q}H^{q-l}(U).$$ Consider the maps $$E_1^{-n-l-1,n+l}\to E_1^{-n-l,n+l}\to E_1^{-n-l+1,n+l}, $$ which correspond to $$\HH^{-1}(X, \DR(\gr^W_{n+l+1}\OX(*D))) \to \HH^{0}(X, \DR(\gr^W_{n+l}\OX(*D))) \to \HH^{1}(X, \DR(\gr^W_{n+l-1}\OX(*D))).$$ By the degeneration of the Hodge-to-de-Rham spectral sequence, we have $$\gr^F_{-n}\HH^{i}(X, \DR(\gr^W_{n+k}\OX(*D))) \cong H^{i}(X, C_k)$$ (see for example \cite{hodgeideals}*{Example 4.2}). Thus, we obtain that the corresponding $\gr^F_{-n}$ piece of the maps above are $$H^{-1}(X, C_{l+1})=0 \to H^0(X, C_l) \to H^1(X, C_{l-1}).$$

For $l\geq 3$, $H^1(X, C_{l-1}) = 0$ by \remarkref{support}. Hence, $$H^0(X, C_l) \cong \gr^F_{-n}E_2^{-n-l,n+l} \cong H^{n,l}(H^n(U)).$$ There is an isomorphism $$ H^{n,l}(H^n(U)) \cong H^{0,n-l}(H_c^n(U))^*$$ (see \cite{PS}*{Theorem 6.23}). Using the long exact sequence of the pair $(X,D)$, $$H^{n-1}(X) \to H^{n-1}(D) \to H_c^n(U) \to H^n(X),$$ we obtain that $$H^{0,n-l}(H_c^n(U)) \cong H^{0,n-l}(H^{n-1}(D)).$$ As the morphism $g: \tilde{D} \to D$ has discriminant $\{x\}$, we have the following sequence of mixed Hodge structures $$H^{n-2}(\tilde{D})\to H^{n-2}(G)\to H^{n-1}(D)\to H^{n-1}(\tilde{D})$$ \cite{PS}*{Corollary 5.37}. We then obtain that $$H^{0,n-l}(H^{n-1}(D)) \cong H^{0,n-l}(H^{n-2}(G)).$$ Hence $$\dim(C_l)_x = h^0(X, C_l) = h^{0,n-l}(H^{n-2}(G)).$$

For $l=2$, consider the following morphisms $$ \gr^F_{-n}E_1^{-n-2,n+2} \to \gr^F_{-n}E_1^{-n-1,n+2} \to \gr^F_{-n}E_1^{-n,n+2},$$ which correspond to $$H^0(X, C_2) \stackrel{\beta}{\to} H^1(X, C_1) \stackrel{\gamma}{\to} H^2(X, C_0).$$ We have \begin{equation*}\begin{split} \ker{\beta}&\cong H^{n,2}(H^n(U))\cong H^{0,n-2}(H_c^n(U))^* \\
\ker{\gamma}/\im{\beta} &\cong H^{n,2}(H^{n+1}(U))\cong H^{0,n-2}(H_c^{n-1}(U))^*\\
\coker{\gamma} &\cong H^{n,2}(H^{n+2}(U))\cong H^{0,n-2}(H_c^{n-2}(U))^* \end{split}\end{equation*}
Consider the following three exact sequences: \begin{equation*}\begin{split} 0\to \ker{\beta} &\to H^0(X,C_2) \to \im{\beta}\to 0 \\ 0\to \ker{\gamma} &\to H^1(X,C_1) \to \im{\gamma}\to 0 \\ 0\to \im{\gamma} &\to H^2(X, C_0) \to \coker{\gamma} \to 0 \end{split}\end{equation*}

Combining the dimensions of the short exact sequences and the isomorphisms above, we obtain \begin{equation*}\begin{split}h^0(X, C_2) &= h^{0,n-2}(H_c^{n-2}(U))-h^{0,n-2}(H_c^{n-1}(U))+h^{0,n-2}(H_c^{n}(U))\\ &+ h^1(X, C_1) - h^2(X, C_0). \end{split}\end{equation*}
From the long exact sequence of the pair $(X,D)$, we obtain \begin{equation*}\begin{split} 0 &=h^{0,n-2}(H_c^{n-2}(U)) - h^{0,n-2}(X) + h^{0,n-2}(H^{n-2}(D)) - h^{0,n-2}(H_c^{n-1}(U))\\ &- h^{0,n-2}(H^{n-1}(D)) + h^{0,n-2}(H_c^{n}(U))\end{split}\end{equation*} Combining the two equations we obtain: \begin{equation*}\begin{split} h^0(X, C_2) &= h^{0,n-2}(\tilde{D}) -h^{0,n-2}(H^{n-2}(D))+h^{0,n-2}(H^{n-1}(D))\end{split}\end{equation*} as $C_1\cong f_*\omega_{\tilde{D}}$ and $C_0\cong \omega_X$. Finally, using that the map $g: \tilde{D}\to D$ has discriminant $\{x\}$, we obtain an exact sequence $$0\to H^{0,n-2}(H^{n-2}(D)) \to H^{0,n-2}(\tilde{D}) \to H^{0,n-2}(H^{n-2}(G))\to H^{0,n-2}(H^{n-1}(D)) \to 0,$$ hence $$\dim{(C_2)_x} = h^0(X,C_2) = h^{0,n-2}(H^{n-2}(G)).$$
\end{proof}

\corollaryref{corcn} and \corollaryref{corc2} are immediate consequences of \theoremref{main}.
\begin{proof}[Proof of \corollaryref{corcn} and \corollaryref{corc2}.] Using the description of Section \ref{mhssnc}, we obtain that $$H^{0,n-2}(H^{n-2}(G))\cong H^{0,n-2}(G(1))$$ by dimension reasons.\\

We have an isomorphism $$W_0H^{n-2}(G) \cong H^{n-2}(\Delta(G))$$ (see \cite{payne13}*{Section 3}). Hence, $$h^{0,0}(H^{n-2}(G)) = \dim{W_0H^{n-2}(G)} = b_{n-2}(\Delta(G)).$$
\end{proof}

\subsection{Example: Isolated weighted homogeneous singularities} For a pair $(X,D)$ such that $D$ has at most isolated weighted homogeneous singularities we can describe the weighted Hodge ideals completely. In this case, there are only at most two possibly different weighted multiplier ideals.

\begin{proposition}\label{whs} Suppose $D$ has only weighted homogeneous isolated singularities, then $$I_0^{W_2}(D) = I_0(D).$$ \end{proposition}

\begin{proof} Applying \theoremref{main}, it is enough to show that given a log-resolution of singularities $g:\tilde{D}\to D$ with exceptional divisor $G\subseteq\tilde{D}$, $H^{n-2}(G)$ has a pure Hodge structure. Let $(Y,0)$ be a germ of one the isolated singularities as above, and $g': \tilde{Y}\to Y$ a log-resolution of singularities induced by $g$, so that the exceptional divisor of $g'$ is $G$ as well. We have an exact sequence of mixed Hodge structures \begin{equation} 0 \to H^{n-2}_G(\tilde{Y})\to H^{n-2}(G)\to H^{n-1}_0(Y)\to 0 \end{equation} and $H^{n-2}_G(\tilde{Y})$ has a pure Hodge structure of weight $n-2$ (see \cite{St83}*{Corollary 1.12}). Moreover, there is an isomorphism of mixed Hodge structures $H^{n-1}_0(Y)(n-1) \simeq (H^n_0(Y))^{\vee}$ \cite{St83}*{Corollary 1.15}. Finally, $H^n_0(Y)$ has a pure Hodge structure of weight $n$ (see for instance \cite{dimca90}*{p. 289}), and therefore $H^{n-1}_0(Y)$ has a pure Hodge structure of weight $n-2$.
\end{proof}

Saito proved a formula for the Hodge ideals of the pair $(X,D)$ \cite{saito09}*{Theorem 0.7}. Suppose that around a singular point $x$, $D$ is defined by $f= \sum{a_mx^m}$. Let $w=(w_1,\ldots, w_n)$ be the weights, so that $\langle w,m \rangle =1$ for all the nonzero summands of $f$. Then, in the case of $I_0(D)$ the formula is given by $$I_0(D) = \langle x^l\ : \langle w, l + \mathbf{1} \rangle \geq 1 \rangle$$ around the point, where $\mathbf{1}= (1,\ldots, 1)$. A similar formula is satisfied by the adjoint ideal of the pair $(X,D)$, by a description of the microlocal $V$-filtration of $\shO_X$ and \cite{saito93}*{\textsection 2}: $$\adj(D) = \langle x^l\ : \langle w, l +\mathbf{1} \rangle > 1 \rangle .$$

\section{Vanishing theorems} Vanishing theorems are a common source of applications for multiplier ideals and Hodge ideals. We discuss vanishing theorems satisfied by weighted multiplier ideals and applications derived from them.

\subsection{Kodaira-type vanishing} The weighted multiplier ideals satisfy a vanishing theorem, which in the case of $I_0(D)$ corresponds to Nadel vanishing. It is a consequence of the fact that weighted multiplier ideals are defined using the lowest piece of the Hodge filtration of $W_{n+l}\OX(*D)$.
\begin{proposition}\label{kodairatypevanishing} Let $X$ be a smooth projective variety and $L$ is an ample line bundle on $X$. Then $$H^i(X, \omega_X(D)\otimes L\otimes I_0^{W_l}(D)) = 0$$ for $i\geq 1.$ \end{proposition}

\begin{remark} In the case of $I_0^{W_1}(D) = \adj(D)$, this result follows from Kawamata-Viehweg Vanishing (see for instance \cite{lazarsfeld1}*{Theorem 4.2.1 and Theorem 4.3.1}). \end{remark}

\begin{proof} We apply the following vanishing result by Saito \cite{saito90}*{Proposition 2.33}: $$\HH^i(X, \gr^F_k\DR(\Mmod)\otimes L) = 0$$ for all $i\geq 1$ and $\Mmod$ a $\Dmod_X$-module corresponding to a mixed Hodge module. We apply it to $W_{n+l}\OX(*D)$ and obtain $$\HH^i(X, \gr^F_{-n}\DR(W_{n+l}\OX(*D))\otimes L) \cong H^i(X, \omega_X(D)\otimes L \otimes I_0^{W_l}(D)) = 0$$ for all $i\geq 1$. \end {proof}

\subsection{Ample divisors} Suppose $X$ is a smooth projective variety of dimension $n$, and $D$ an ample divisor. Let $U:= X\smallsetminus D$. As $U$ is a smooth and affine variety, $H^{i+n}(U) = 0$ for $i> 0$ (see for instance \cite{lazarsfeld2}*{Theorem 3.1.1}). Using this, one can deduce that $$H^i(X, \omega_X(D) \otimes I_0(D)) = 0$$ for all $i\geq 1$ \cite{hodgeideals}*{Proposition 23.1}. We discuss the obstruction for an analogous vanishing theorem for weighted Hodge ideals. \\

\begin{proposition} \label{vanishingample} There is a short exact sequence: $$0\to H^i(X, \omega_X(D)\otimes I_0^{W_l}(D)) \to H^i(X,C_l)\to H^{i+1}(X,\omega_X(D)\otimes I_0^{W_{l-1}}(D)) \to 0$$ for $i\geq 1$. \end{proposition}

\begin{remark} Using basic results in mixed Hodge theory and the exact sequence $$0\to\omega_X\to\omega_X(D)\otimes\adj(D)\to f_*\omega_{\tilde{D}}\to 0,$$ one can obtain short exact sequences $$0\to H^i(X, \omega_X(D)\otimes \adj(D)) \to H^i(X, f_*\omega_{\tilde{D}}) \to H^{i+1}(X, \omega_X)\to 0$$ for $i\geq 1$, which corresponds to the case $l=1$ of \propositionref{vanishingample}. This is used to study the geometry of singular points in plane curves in \cite{lazarsfeld10}*{Exercise 3.8}. 
\end{remark}

\begin{proof} We are interested in the spaces $H^i(X, \omega_X(D) \otimes I_0^{W_k}(D))$, which can be described as $$H^i(X, \omega_X(D) \otimes I_0^{W_k}(D)) \cong \gr_{-n}^FH^i(X, \DR(W_{n+k}\OX(*D))).$$  Comparing the spectral sequences (see (\ref{ss2})) $$E_1^{-n-l, q} = H^{q-n-l}(X, \DR(\gr^W_{n+l}\OX(*D))) \Rightarrow H^{q-l}(U, \CC)$$ and $$ E_1^{'-n-l, q} = H^{q-n-l}(X, \DR(\gr^W_{n+l}W_{n+k}\OX(*D))) \Rightarrow H^{q-n-l}(X, \DR(W_{n+k}\OX(*D)))$$ and noting that $$E_{2}^{-n-l,q} \cong \Gr^W_qH^{q-l}(U, \CC)$$ $$E_{2}^{'-n-l,q} \cong \gr^W_qH^{q-n-l}(X, \DR(W_{n+k}\OX(*D)))$$ we obtain the following:
\begin{aenumerate}\item\label{a} For $s\geq 1$, $$\gr^W_{i+n+k-s}H^i(X, \DR(W_{n+k}\OX(*D))) \cong \Gr^W_{i+n+k-s}H^{i+n}(U, \CC).$$
\item\label{b} For $s\geq 1$, $$ \gr^W_{i+n+k+s}H^i(X, \DR(W_{n+k}\OX(*D))) = 0.$$
\item\label{c} Let $$\alpha: H^{i-1}(X, \DR(\gr^W_{n+k+1}\OX(*D))) \to H^i(X, \DR(\gr^W_{n+k}\OX(*D)))$$ corresponding to the map $E^{-n-k-1, i+n+k}_1 \to E_1^{-n-k, i+n+k}$. Then we have the following short exact sequence: $$0\to \im{\alpha} \to \gr^W_{i+n+k}H^i(X, \DR(W_{n+k}\OX(*D))) \to \Gr^W_{i+n+k}H^{i+n}(U, \CC) \to 0.$$
\end{aenumerate}

For $i\geq 1$ we obtain $$H^i(X, \DR(W_{n+k}\OX(*D))) \cong \gr^W_{i+n+k}H^i(X, \DR(W_{n+k}\OX(*D))) \cong \im{\alpha}.$$

Using \ref{c} for $k=l-1$, we obtain that $$\im{\alpha} \cong H^{i+1}(X, \DR(W_{n+l-1}\shO_X(*D)))$$ for $$\alpha: H^i(X, \DR(\gr^W_{n+l}\shO_X(*D))) \to H^{i+1}(X, \DR(\gr^W_{n+l-1}\shO_X(*D))).$$ We apply $\gr^F_{-n}\DR$ and obtain the sequence: $$0\to\ker\gr^F_{-n}\alpha\to H^i(X, C_l) \to H^{i+1}(X, \omega_X(D)\otimes I_0^{W_{l-1}}(D)) \to 0.$$ Finally, we have $$E_2^{-n-l, n+l+i} \cong \Gr^W_{n+l+i} H^{n+i}(U,\CC) = 0.$$ Therefore, $$  \ker\gr^F_{-n}\alpha \cong H^i(X, \omega_X(D)\otimes I_0^{W_l}(D)).$$ The result follows.
\end{proof}

\begin{proof}[Proof of \theoremref{thmvanishingample}] If $D$ has isolated singularities, then $C_l$ is a skyscraper sheaf for $l\geq 2$ and hence $$H^i(X, C_l)=0$$ for $i\geq 1$. Applying \propositionref{vanishingample} we obtain the result.\\
If $\dim{X} = 2$, then $H^1(X, f_*\omega_{\tilde{D}}) \cong H^{1,1}(\tilde{D})$ that is 1-dimensional. Similarly, $H^2(X, \omega_X)\cong H^{2,2}(X)$ that is 1-dimensional as well. Hence, $$H^1(X, \omega_X(D)\otimes I_0^{W_1}(D)) = 0.$$
\end{proof}

\begin{remark}\label{remarkvanishing} This result does not hold in general for $l=1$ and $i=1$. Indeed, let $X= \PP^3$. The sequence in \propositionref{vanishingample} is $$0\to H^1(X, \omega_X(D)\otimes I_0^{W_1}(D)) \to H^1(\tilde{D},\omega_{\tilde{D}})\to H^2(X,\omega_X )=0,$$ hence $$H^1(X, \omega_X(D)\otimes I_0^{W_1}(D)) \cong H^{2,1}(\tilde{D}).$$ Let $D = (x_0^3 + x_1^3+x_2^3=0)$. $D$ is the projective cone of a smooth projective curve $C$ of genus 1. A log-resolution of the pair $(X,D)$ is the blowup at the vertex of $D$. The cohomology $$H^{2,1}(\tilde{D})\cong H^{1,0}(C)\otimes H^{1,1}(\PP^1) \cong \CC.$$
\end{remark}

\subsection{Applications} In this section, we combine the local and global results of weighted multiplier ideals. First, we prove \corollaryref{corind}, which is a direct consequence of vanishing theorems we have proved.
\begin{proof}[Proof of \corollaryref{corind}] Consider the short exact sequence $$0\to \shO_{\PP^n}(k)\otimes I_0^{W_l}(D) \to \shO_{\PP^n}(k) \to \shO_{Z_l}(k)\to 0.$$ The result follows from taking the long exact sequence of cohomology and applying \theoremref{thmvanishingample} and \propositionref{kodairatypevanishing}. \end{proof}

\noindent\textbf{Isolated log-canonical singularities.} When the pair $(X,D)$ is log-canonical, and $D$ has at most isolated singularities, for a singular point $x\in D$ there is at most one $l$ such that $$\dim{(C_l)_x} =1$$ and the rest of the dimensions of $(C_k)_x$ are 0 \cite{ishii85}*{Proposition 3.7}. Moreover, if the singularity is not rational, such an $l$ exists and $x$ is called a singularity of type $(0,n-l)$ \cite{ishii85}*{Definition 4.1}. In particular, the lowest possible type is $(0,0)$, and the highest is $(0,n-2)$.  What Ishii proved is that $$\dim{\Gr_F^0H^{n-2}(G)} = h^{n-2}(G, \shO_G)= 1,$$ and this result includes that $$p_g(D,x) = 1,$$ where $p_g(D,x) = \dim(R^{n-2}g_*\shO_{\tilde{D}})_x$ is the geometric genus of the singularity. The geometric genus can be computed as $$p_g(D,x) = \dim(\omega_D/g_*\omega_{\tilde{D}})_x$$ (see \cite{ishiibook}*{Proposition 6.3.11}). Using the sequences $$0\to \omega_X\to \omega_X(D) \otimes \adj(D) \to f_*\omega_{\tilde{D}}\to 0,$$ and $$0\to \omega_X\to \omega_X(D) \to \omega_{D}\to 0,$$ we see that if $Z_1$ is the scheme defined by $\adj(D)$, then the length of $\shO_{Z_1}$ at $x$ is $$\dim(\omega_D/g_*\omega_{\tilde{D}})_x = p_g(D,x).$$ This means that as $p_g = 1$ for all the log-canonical not rational singularities, $Z_1$ is a reduced scheme. Using this we obtain the following result.

\begin{corollary}\label{proplcpoints} Let $D$ a reduced hypersurface of $\PP^n$ of degree $d$ with at most isolated singularities. Assume $(\PP^n,D)$ is log-canonical. Let $Z$ be the union of the singular points in $D$ which are not rational, and are of type $(0,0)$, \ldots, $(0, n-3)$. Then, $$\#{Z} \leq \binom{d-1}{n}.$$ \end{corollary}

\begin{proof} By the discussion above, $Z$ coincides with the scheme defined by $I_0^{W_2}(D)$. The result follows by applying \corollaryref{corind}. \end{proof}

\begin{remark}In the same setting as \corollaryref{proplcpoints}, using the vanishing theorem for $\adj(D)$, we obtain that the number of non-rational points is bounded by $\binom{d}{n}$. Thus, the weighted multiplier ideals give a better bound on this subset. Note that as $I_0(D)$ is trivial, we cannot use Nadel vanishing theorem.
\end{remark}

\noindent\textbf{Isolated non log-canonical singularities.} Let $D\subseteq \PP^n$ be a reduced hypersurface of degree $d$ with at most isolated singularities. Using the strategy of the proof of \corollaryref{proplcpoints}, when $D$ has low degree, we still obtain information about the singularities of $D$ without assuming that the pair $(\PP^n, D)$ is log-canonical. More precisely, we obtain that if $d=n$, then $D$ has at most one non-rational singularity, and this singular point is log-canonical of type $(0,n-2)$. Moreover, let $Z_l$ be the scheme corresponding to $I_0^{W_l}(D)$. We also obtain that if $d=n+1$, then $Z_2$ (and the rest of $Z_l$ for $l\geq 3$) has at most length one. This means that all non-rational singular points, except possibly one, are log-canonical of type $(0, n-2)$.\\

If we use the same methods in the case $l=1$, we obtain that $Z_1$ has length $q+1$, and $q\leq n$. In particular, the sum of the geometric genera of the singular points is equal to $q+1$. The number of singular points depends on the geometric genus of the one singular point that is not log-canonical of type $(0,n-2)$ (as the other singular points have geometric genus equal to 1). For instance, if $q=1$, then the non-rational singular points of $D$ are either one point of geometric genus 2, or two points of geometric genus 1, one of which must be log-canonical of type $(0,n-2)$.\\

Indeed, consider as in the proof of \corollaryref{proplcpoints}, the exact sequence \begin{equation}\begin{split} 0 & \to H^0(\PP^n, \shO_{\PP^n}(d-n-1)\otimes I^{W_l}_0(D)) \to H^0(\PP^n, \shO_{\PP^n}(d-n-1))\to H^0(\PP^n, \shO_{Z_l}) \\ & \to H^1(\PP^n, \shO_{\PP^n}(d-n-1)\otimes I^{W_l}_0(D)) \to 0.  \end{split}\end{equation} Recall that by \theoremref{thmvanishingample},  $$H^1(\PP^n, \shO_{\PP^n}(d-n-1)\otimes I^{W_l}_0(D)) =0$$ for $l\geq 2$, and using a similar argument to the one used in \remarkref{remarkvanishing}, $$H^1(\PP^n, \shO_{\PP^n}(d-n-1)\otimes I^{W_1}_0(D)) \simeq H^{n-1,1}(\tilde{D}),$$ where $\tilde{D}\to D$ is a log-resolution of singularities. Let $q:= h^{n-1,1}(\tilde{D}) = h^{0,n-2}(\tilde{D})$. Note that this number does not depend on the log-resolution. We also have a short exact sequence \begin{equation} 0 \to H^0(\PP^n, \shO_{\PP^n}(d-n)\otimes I^{W_1}_0(D)) \to H^0(\PP^n, \shO_{\PP^n}(d-n))\to H^0(\PP^n, \shO_{Z_1})\to 0.\end{equation} The results follows from combining the information of these exact sequences.\\

In the case of cubic surfaces, this recovers a result of Bruce and Wall, which is a consequence of the classification of singularities in \cite{bw79}. For quartic surfaces, we partially recover the results in \cite{umezu81}*{Theorem 1}.

\section{Restrictions of weighted multiplier ideals to subvarieties} In this section, we describe the relation between the weighted multiplier ideals of a pair, and the pair obtained by restricting to a subvariety consisting of the intersection of general hypersurfaces. 

\subsection{Restriction Theorem} We start by comparing the weighted multiplier ideals of a pair $(X,D)$, and the pair arising from restricting to a general hypersurface.

\begin{proof}[Proof of \theoremref{thmrestriction}.]Let $f: Y\to X$ be a log-resolution of the pair $(X,D)$, $E= f^{-1}(D)_{red}$, and $H':= f^{-1}(H)$. We will use the following short exact sequence: $$0\to W_l\omega_Y(E) \to W_l\omega_Y(E)\otimes\shO_Y(H') \to W_l\omega_{H'}(E\restr{H'}) \to 0.$$ 
Pushing forward and tensoring by $\omega_X(D+H)^{-1}$, using that $\omega_X(D+H)\otimes\shO_H \cong \omega_H(D_H)$, we obtain the following long exact sequence: \begin{equation*}\begin{split}& 0\to  I_0^{W_l}(D)\otimes\shO_X(-H) \to I_0^{W_l}(D)\to I_0^{W_l}(D_H)\to\\\to &R^1f_*W_l\omega_Y(E)\otimes\shO_X(-H)\otimes\omega_X(D)^{-1}\to R^1f_*W_l\omega_Y(E)\otimes\omega_X(D)^{-1}.\end{split}\end{equation*} The kernel of the last map is isomorphic to $\shtor_1(R^1f_*W_l\omega_Y(E)\otimes\omega_X^{-1}, \shO_H) \otimes \OX(-D)$, hence we obtain an exact sequence \begin{equation}\begin{split}\label{sequencerestriction} 0\to I_0^{W_l}(D)\otimes\shO_X(-H) \to I_0^{W_l}(D)\to I_0^{W_l}(D_H)\to\\ \to \shtor_1(R^1f_*W_l\omega_Y(E)\otimes\omega_X^{-1}, \shO_H) \otimes \OX(-D)\to 0. \end{split} \end{equation}

As $H$ is general, the inclusion $H\into X$ is non-characteristic for the Hodge module which corresponds to the $\Dmod_X$-module $H^1f_+(W_{l+n}\shO_Y(*E)) =: M$. This means that $L^ji^*\gr_p^FM =0$ for all $j<0$, and $p\in\ZZ$ (see \cite{schnell16}*{Section 8}). As $$R^1f_*W_l\omega_Y(E) \cong \gr_{-n}^F\DR(M) \cong \omega_X\otimes\gr_0^F(M)$$ by \lemmaref{interpretation}, we obtain that $$\shtor_1(R^1f_*W_l\omega_Y(E)\otimes\omega_X^{-1}, \shO_H) = 0.$$ Thus, using (\ref{sequencerestriction}) we obtain $$I_0^{W_l}(D)\cdot \shO_H = I_0^{W_l}(D_H).$$
\end{proof}

\begin{remark}\label{remarkrestriction} Weighted multiplier ideals, in fact, satisfy a stronger result. This result is the analogue of the Restriction Theorem for multiplier ideals \cite{lazarsfeld2}*{Theorem 9.5.1}. Given any hypersurface $H\subseteq X$ such that $H\subsetneq \Supp(D)$, and $D\restr{H}=:D_H$ is reduced, \begin{equation}\label{rt} I_0^{W_l}(D_H)\subseteq I_0^{W_l}(D)\cdot\shO_H\end{equation} for all $l\geq 0$. The proof of this result requires the methods of mixed Hodge modules not described in this paper and holds for all weighted Hodge ideals. For this reason, it is presented in \cite{olano21b}. When $l=1$, (\ref{rt}) can be proved using a similar proof to the one used for multiplier ideals.
\end{remark} 

\begin{remark} The generality assumption in \theoremref{thmrestriction} is used to ensure that $H$ is smooth, $D\restr{H}$ is reduced, and given $f:Y\to X$ a log-resolution of the pair $(X,D)$, then it is also a log-resolution of the pair $(X, D+H)$ and $f^*H = H'$, the strict transform of $H$. These conditions are satisfied if we work we several general hypersurfaces simultaneously, and in this way we obtain a more general version of \theoremref{thmrestriction}. Namely, let $H_1, \ldots , H_s$ general elements of base-point free linear systems $L_1, \ldots ,L_s$. Let $V = \bigcap H_i$ and $l\geq 0$, then $$I_0^{W_l}(D\restr{V}) = I_0^{W_l}(D)\cdot \shO_V.$$
\end{remark}

As an application of \theoremref{thmrestriction} and its more general version, we obtain a description of the rank of $C_l$ for a pair $(X,D)$. Let $s$ be the dimension of $D_{\rm sing}$, $H_1, \ldots , H_s$ general hypersurfaces, and $V = \bigcap H_i$. By construction, $D\restr{V} =: D_V$ has at most isolated singularities. Let $\tilde{V}\to V$ a log-resolution of singularities of the pair $(V,D_V)$, and $G_V\subseteq \widetilde{D_V}$ the exceptional divisor of the restriction of this map to $\widetilde{D_V}$.

\begin{proposition}\label{rankcl} Using the notation above, let $Z\subseteq D_{\rm sing}$ be an irreducible component of dimension $s$, $x\in Z$ a singular point of $D_V$, and $G_{V,x}$ be the fiber of $x$ on the resolution of singularities of $D_V$. Then, the rank of $C_l\restr{Z}$ as coherent sheaf on $Z$ is $h^{0,n-l-s}(H^{n-2-s}(G_{V,x}))$. \end{proposition}

\begin{proof} We denote by $C_{l,V}$ the sheaves defined in \definitionref{differencesheaves} for the pair $(V, D_V)$. By \theoremref{main}, the dimension of $C_{l,V}$ at $x$ is $h^{0,n-l-s}(H^{n-2-s}(G_{V,x}))$. Consider the diagram 
\begin{diagram*}{2.5em}{2em}
\matrix[math] (m) {  &  I_0^{W_{l-1}}(D)\otimes\shO_V & I_0^{W_{l}}(D)\otimes\shO_V  & \tilde{C_l}\otimes\shO_V & 0 \\
										0 &  I_0^{W_{l-1}}(D_V) & I_0^{W_{l}}(D_V) & \widetilde{C_{l,V}} & 0\\}; 
 \path[to] (m-1-2) edge (m-1-3);
 \path[to] (m-1-3) edge (m-1-4);
 \path[to] (m-1-4) edge (m-1-5);
 \path[to] (m-2-1) edge (m-2-2);
 \path[to] (m-2-2) edge (m-2-3);
 \path[to] (m-2-3) edge (m-2-4);
 \path[to] (m-2-4) edge (m-2-5);
 \path[to] (m-1-2) edge (m-2-2);
 \path[to] (m-1-3) edge (m-2-3);
 \path[to] (m-1-4) edge node[auto]{$\alpha$} (m-2-4);

\end{diagram*}
where the first two vertical arrows are given by \theoremref{thmrestriction}. By the Snake Lemma, the map $\alpha$ is a surjection. This map is induced by the multiplication map, hence $\tilde{C_l}\cdot\shO_V=\widetilde{C_{l,V}}$. Let $\tilde{C_l}\onto C'_l$ be the surjection onto the torsion-free part. As $V$ is the intersection of $s$ general hypersurfaces, $\tilde{C_l}\cdot\shO_V = C'_l\cdot\shO_V$, which is a skyscraper sheaf of length (or dimension) at a point $x$ as described above. This length coincides with the rank of $\tilde{C_l}$, and the latter has the same rank as $C_l$.\end{proof}

Using \theoremref{thmrestriction}, we also obtain information about the singularities of some hypersurfaces of $\PP^n$ with non-isolated singularities. Before stating the precise result, recall the information that can be obtained using the ideals $I_0(D)$ and $\adj(D)$. Let $D\subseteq \PP^n$ be a hypersurface of degree $d$. Let $Z_n$ and $Z_1$ be the schemes defined by the ideal $I_0(D)$ and $\adj(D)$, and let $z_n, z_1$ be their corresponding dimensions. If $z_n \leq n-d$, then the pair $(\PP^n, D)$ is log-canonical. Moreover, if $z_1\leq n-d-1$, then the pair $(\PP^n,D)$ has rational singularities. Indeed, we can assume $z_n=0$ using \theoremref{thmrestriction}, then the condition on the degree of $D$ is that $n-d\geq 0$. Using the exact sequence $$0\to \omega_{\PP^n}(D)\otimes I_0(D)\to \omega_{\PP^n}(D) \to \omega_{\PP^n}(D)\otimes\shO_{Z_n}\to 0,$$ the fact that $\omega_{\PP^n}(D) = \shO_{\PP^n}(-n-1+d)$, and Nadel vanishing, we obtain the result. The analogous result for $Z_1$ follows from a similar argument, replacing $\omega_{\PP^n}(D)$ in the sequence by $\omega_{\PP^n}(D+H)$, where $H\subseteq \PP^n$ is a degree 1 hypersurface, and applying \propositionref{kodairatypevanishing} instead of Nadel vanishing.\\

For the schemes $Z_l$ defined by the ideals $I_0^{W_l}(D)$, we obtain the same result than for $Z_n$.
\begin{proposition} Using the notation above, let $z_l=\dim{Z_l}$ and suppose $z_l\leq n-d$. Then $Z_l=\emptyset$. In particular, if $z_1 = n-d$, then $Z_2 = \emptyset$. Moreover, in this case, if $Z_1\neq\emptyset$, that is, if $D$ has non-rational singularties, then these singularities are \textit{compound} $(0, n-z_1-2)$ singularities. \end{proposition}

The first part of the statement follows from the same sequence as above and \theoremref{thmvanishingample}. In the second part, by \textit{compound} $(0,n-z_1-2)$ singularities, we mean that after intersecting with a linear subspace $L$ of dimension $n-z_1$ (given by the intersection of $z_1$ general hyperplane sections), $D\restr{L} =:D_L$ has only isolated log-canonical singularities, and these singularities are of type $(0, n-z_1-2) = (0, \dim{L} -2)$. For instance, if $n=4$, $d=3$ and $z_1=1$, that is, $D$ is a hypersurface of $\PP^4$ of degree 3 with a 1-dimensional non-rational singular locus, then the transversal singularity type is simple elliptic. The second statement follows from applying \theoremref{thmrestriction}.

\bibliography{../../bib}

\end{document}

%% file: macros.tex
\usepackage{amsmath,amsfonts,amsthm,mathrsfs}
\usepackage{amssymb}

\usepackage[unicode,bookmarks]{hyperref}
\usepackage[usenames,dvipsnames]{xcolor}
\hypersetup{colorlinks=true,citecolor=NavyBlue,linkcolor=Brown,urlcolor=Orange}

\usepackage[alphabetic,initials]{amsrefs}

\usepackage{enumitem}
\newenvironment{renumerate}{%
	\begin{enumerate}[label=(\roman{*}), ref=(\roman{*})]
}{%
	\end{enumerate}%
}

\newenvironment{aenumerate}{%
	\begin{enumerate}[label=(\alph{*}), ref=(\alph{*})]
}{%
	\end{enumerate}%
}

\usepackage{chngcntr}

\ifdraft
\usepackage[notcite,notref,color]{showkeys}

\definecolor{labelkey}{gray}{0.5}
\fi

\usepackage{tikz}

\usepackage{tikz-cd}

\usetikzlibrary{matrix,arrows}
\newlength{\myarrowsize} 

\pgfarrowsdeclare{cmto}{cmto}{
	\pgfsetdash{}{0pt} 
	\pgfsetbeveljoin 
	\pgfsetroundcap 
	\setlength{\myarrowsize}{0.6pt}
	\addtolength{\myarrowsize}{.5\pgflinewidth}
	\pgfarrowsleftextend{-4\myarrowsize-.5\pgflinewidth} 
	\pgfarrowsrightextend{.8\pgflinewidth}
}{
	\setlength{\myarrowsize}{0.6pt} 
  	\addtolength{\myarrowsize}{.5\pgflinewidth}  
	\pgfsetlinewidth{0.5\pgflinewidth}
	\pgfsetroundjoin
	\pgfpathmoveto{\pgfpoint{1.5\pgflinewidth}{0}}
	\pgfpatharc{-109}{-170}{4\myarrowsize}
	\pgfpatharc{10}{189}{0.58\pgflinewidth and 0.2\pgflinewidth}
	\pgfpatharc{-170}{-115}{4\myarrowsize+\pgflinewidth}
	\pgfpathclose
	\pgfusepathqfillstroke
	\pgfpathmoveto{\pgfpoint{1.5\pgflinewidth}{0}}
	\pgfpatharc{109}{170}{4\myarrowsize}
	\pgfpatharc{-10}{-189}{0.58\pgflinewidth and 0.2\pgflinewidth}
	\pgfpatharc{170}{115}{4\myarrowsize+\pgflinewidth}
	\pgfpathclose
	\pgfusepathqfillstroke
	\pgfsetlinewidth{2\pgflinewidth}
}

\pgfarrowsdeclare{cmonto}{cmonto}{
	\pgfsetdash{}{0pt} 
	\pgfsetbeveljoin 
	\pgfsetroundcap 
	\setlength{\myarrowsize}{0.6pt}
	\addtolength{\myarrowsize}{.5\pgflinewidth}
	\pgfarrowsleftextend{-4\myarrowsize-.5\pgflinewidth} 
	\pgfarrowsrightextend{.8\pgflinewidth}
}{
	\setlength{\myarrowsize}{0.6pt} 
  	\addtolength{\myarrowsize}{.5\pgflinewidth}  
	\pgfsetlinewidth{0.5\pgflinewidth}
	\pgfsetroundjoin
	\pgfpathmoveto{\pgfpoint{1.5\pgflinewidth}{0}}
	\pgfpatharc{-109}{-170}{4\myarrowsize}
	\pgfpatharc{10}{189}{0.58\pgflinewidth and 0.2\pgflinewidth}
	\pgfpatharc{-170}{-115}{4\myarrowsize+\pgflinewidth}
	\pgfpathclose
	\pgfusepathqfillstroke
	\pgfpathmoveto{\pgfpoint{1.5\pgflinewidth}{0}}
	\pgfpatharc{109}{170}{4\myarrowsize}
	\pgfpatharc{-10}{-189}{0.58\pgflinewidth and 0.2\pgflinewidth}
	\pgfpatharc{170}{115}{4\myarrowsize+\pgflinewidth}
	\pgfpathclose
	\pgfusepathqfillstroke
	\pgfpathmoveto{\pgfpoint{1.5\pgflinewidth-0.3em}{0}}
	\pgfpatharc{-109}{-170}{4\myarrowsize}
	\pgfpatharc{10}{189}{0.58\pgflinewidth and 0.2\pgflinewidth}
	\pgfpatharc{-170}{-115}{4\myarrowsize+\pgflinewidth}
	\pgfpathclose
	\pgfusepathqfillstroke
	\pgfpathmoveto{\pgfpoint{1.5\pgflinewidth-0.3em}{0}}
	\pgfpatharc{109}{170}{4\myarrowsize}
	\pgfpatharc{-10}{-189}{0.58\pgflinewidth and 0.2\pgflinewidth}
	\pgfpatharc{170}{115}{4\myarrowsize+\pgflinewidth}
	\pgfpathclose
	\pgfusepathqfillstroke
	\pgfsetlinewidth{2\pgflinewidth}
}

\pgfarrowsdeclare{cmhook}{cmhook}{
	\pgfsetdash{}{0pt} 
	\pgfsetbeveljoin 
	\pgfsetroundcap 
	\setlength{\myarrowsize}{0.6pt}
	\addtolength{\myarrowsize}{.5\pgflinewidth}
	\pgfarrowsleftextend{-4\myarrowsize-.5\pgflinewidth} 
	\pgfarrowsrightextend{.8\pgflinewidth}
}{
	\setlength{\myarrowsize}{0.6pt} 
  	\addtolength{\myarrowsize}{.5\pgflinewidth}  
 	\pgfsetdash{}{0pt}
	\pgfsetroundcap
	\pgfpathmoveto{\pgfqpoint{0pt}{-4.667\pgflinewidth}}
	\pgfpathcurveto
    {\pgfqpoint{4\pgflinewidth}{-4.667\pgflinewidth}}
    {\pgfqpoint{4\pgflinewidth}{0pt}}
    {\pgfpointorigin}
	\pgfusepathqstroke
}

\newenvironment{diagram*}[2]{%
\[%
\begin{tikzpicture}[>=cmto,baseline=(current bounding box.center),%
	to/.style={->,font=\scriptsize,cap=round},%
	into/.style={cmhook->,font=\scriptsize,cap=round},%
	onto/.style={-cmonto,font=\scriptsize,cap=round},%
	math/.style={matrix of math nodes, row sep=#2, column sep=#1,%
		text height=1.5ex, text depth=0.25ex}]%
}{%
\end{tikzpicture}%
\]%
\ignorespacesafterend%
}

\newcommand{\Dmod}{\mathscr{D}}
\newcommand{\Mmod}{\mathcal{M}}

\newcommand{\shT}{\mathscr{T}}

\newcommand{\cohH}{\mathcal{H}}

\newcommand{\shtor}{\shT\hspace{-2.5pt}\mathit{or}}

\newcommand{\ZZ}{\mathbb{Z}}
\newcommand{\QQ}{\mathbb{Q}}

\newcommand{\CC}{\mathbb{C}}
\newcommand{\HH}{\mathbb{H}}
\newcommand{\PP}{\mathbb{P}}

\DeclareMathOperator{\im}{im}

\DeclareMathOperator{\coker}{coker}

\DeclareMathOperator{\Supp}{Supp}

\DeclareMathOperator{\gr}{gr}
\DeclareMathOperator{\DR}{DR}

\newcommand{\shf}[1]{\mathscr{#1}}
\newcommand{\OX}{\shf{O}_X}

\newcommand{\restr}[1]{\big\vert_{#1}}

\def\overbar#1#2#3{{%
	\setbox0=\hbox{\displaystyle{#1}}%
	\dimen0=\wd0
	\advance\dimen0 by -#2 
	\vbox {\nointerlineskip \moveright #3 \vbox{\hrule height 0.3pt width \dimen0}%
		\nointerlineskip \vskip 1.5pt \box0}%
}}

\newcommand{\into}{\hookrightarrow}
\newcommand{\onto}{\to\hspace{-0.7em}\to}

\newcommand{\shO}{\shf{O}}

\makeatletter
\let\@@seccntformat\@seccntformat
\renewcommand*{\@seccntformat}[1]{%
  \expandafter\ifx\csname @seccntformat@#1\endcsname\relax
    \expandafter\@@seccntformat
  \else
    \expandafter
      \csname @seccntformat@#1\expandafter\endcsname
  \fi
    {#1}%
}
\newcommand*{\@seccntformat@subsection}[1]{%
  \textbf{\csname the#1\endcsname.}
}
\makeatother

\makeatletter
\let\@paragraph\paragraph
\renewcommand*{\paragraph}[1]{%
	\vspace{0.3\baselineskip}%
	\@paragraph{\textit{#1}}%
}
\makeatother

\counterwithin{equation}{subsection}
\counterwithout{subsection}{section}
\counterwithin{figure}{subsection}

\newtheorem*{theorem*}{Theorem}
\newtheorem{lemma}[equation]{Lemma}
\newtheorem*{lemma*}{Lemma}
\newtheorem{corollary}[equation]{Corollary}
\newtheorem{proposition}[equation]{Proposition}
\newtheorem*{proposition*}{Proposition}

\theoremstyle{definition}
\newtheorem{definition}[equation]{Definition}
\newtheorem*{definition*}{Definition}
\newtheorem{remark}[equation]{Remark}

\newtheorem{example}[equation]{Example}
\newtheorem*{example*}{Example}
\newtheorem*{problem*}{Problem}

\theoremstyle{plain}

\newcommand{\theoremref}[1]{\hyperref[#1]{Theorem~\ref*{#1}}}
\newcommand{\lemmaref}[1]{\hyperref[#1]{Lemma~\ref*{#1}}}
\newcommand{\definitionref}[1]{\hyperref[#1]{Definition~\ref*{#1}}}
\newcommand{\propositionref}[1]{\hyperref[#1]{Proposition~\ref*{#1}}}
\newcommand{\conjectureref}[1]{\hyperref[#1]{Conjecture~\ref*{#1}}}
\newcommand{\corollaryref}[1]{\hyperref[#1]{Corollary~\ref*{#1}}}
\newcommand{\exampleref}[1]{\hyperref[#1]{Example~\ref*{#1}}}

\makeatletter
\let\old@caption\caption
\renewcommand*{\caption}[1]{%
	\setcounter{figure}{\value{equation}}%
	\stepcounter{equation}%
	\old@caption{#1}\relax%
}
\makeatother

\newcounter{intro}

\newtheorem{intro-conjecture}[intro]{Conjecture}
\newtheorem{intro-corollary}[intro]{Corollary}
\newtheorem{intro-theorem}[intro]{Theorem}

\newcommand{\omY}{\omega_Y}

\newcommand{\OY}{\shO_Y}

\newcommand{\parref}[1]{\hyperref[#1]{\S\ref*{#1}}}

\makeatletter
\newcommand*\if@single[3]{%
  \setbox0\hbox{${\mathaccent"0362{#1}}^H$}%
  \setbox2\hbox{${\mathaccent"0362{\kern0pt#1}}^H$}%
  \ifdim\ht0=\ht2 #3\else #2\fi
  }
\newcommand*\rel@kern[1]{\kern#1\dimexpr\macc@kerna}
\newcommand*\widebar[1]{\@ifnextchar^{{\wide@bar{#1}{0}}}{\wide@bar{#1}{1}}}
\newcommand*\wide@bar[2]{\if@single{#1}{\wide@bar@{#1}{#2}{1}}{\wide@bar@{#1}{#2}{2}}}
\newcommand*\wide@bar@[3]{%
  \begingroup
  \def\mathaccent##1##2{%
    \if#32 \let\macc@nucleus\first@char \fi
    \setbox\z@\hbox{$\macc@style{\macc@nucleus}_{}$}%
    \setbox\tw@\hbox{$\macc@style{\macc@nucleus}{}_{}$}%
    \dimen@\wd\tw@
    \advance\dimen@-\wd\z@
    \divide\dimen@ 3
    \@tempdima\wd\tw@
    \advance\@tempdima-\scriptspace
    \divide\@tempdima 10
    \advance\dimen@-\@tempdima
    \ifdim\dimen@>\z@ \dimen@0pt\fi
    \rel@kern{0.6}\kern-\dimen@
    \if#31
      \overline{\rel@kern{-0.6}\kern\dimen@\macc@nucleus\rel@kern{0.4}\kern\dimen@}%
      \advance\dimen@0.4\dimexpr\macc@kerna
      \let\final@kern#2%
      \ifdim\dimen@<\z@ \let\final@kern1\fi
      \if\final@kern1 \kern-\dimen@\fi
    \else
      \overline{\rel@kern{-0.6}\kern\dimen@#1}%
    \fi
  }%
  \macc@depth\@ne
  \let\math@bgroup\@empty \let\math@egroup\macc@set@skewchar
  \mathsurround\z@ \frozen@everymath{\mathgroup\macc@group\relax}%
  \macc@set@skewchar\relax
  \let\mathaccentV\macc@nested@a
  \if#31
    \macc@nested@a\relax111{#1}%
  \else
    \def\gobble@till@marker##1\endmarker{}%
    \futurelet\first@char\gobble@till@marker#1\endmarker
    \ifcat\noexpand\first@char A\else
      \def\first@char{}%
    \fi
    \macc@nested@a\relax111{\first@char}%
  \fi
  \endgroup
}
\makeatother